\documentclass[12pt]{amsart}
\usepackage{amssymb,amsmath,bm}
\usepackage{enumitem}
\usepackage[abbrev,lite,nobysame]{amsrefs}
\usepackage{enumitem}
\usepackage{color}
\usepackage{amsfonts}
\usepackage{amsthm}
\usepackage{graphicx}

\newcommand{\ba}{\mathbb A}
\newcommand{\bb}{\mathbb B}

\newcommand{\be}{\mathbb E}
\newcommand{\bk}{\mathbb K}
\newcommand{\bn}{\mathbb N}

\newcommand{\br}{\mathbb R}

\newcommand{\dist}{{\rm dist}}
\DeclareMathOperator*{\esssup}{ess\,sup}

\newcommand{\nc}{\newcommand}
\newtheorem{proposition}{Proposition}[section]
\newtheorem{theorem}[proposition]{Theorem}
\newtheorem{corollary}[proposition]{Corollary}
\newtheorem{lemma}[proposition]{Lemma}
\theoremstyle{definition}
\newtheorem{definition}[proposition]{Definition}
\newtheorem{remark}[proposition]{Remark}
\newtheorem{example}[proposition]{Example}
\numberwithin{equation}{section}

\nc{\cD}{{\cal D}}
\nc{\cP}{{\cal P}}
\nc{\cR}{{\cal R}}
\nc{\eps}{\varepsilon}
\def\e{{\rm e}}
\def\conv{{\rm conv}}
\def\supp{{\rm supp}}
\def\spa{{\rm span}}
\def\d{{\rm d}}
\nc{\Om}{\Omega}
\nc{\om}{\omega}
\nc{\aal}{\alpha}
\nc{\tow}{\rightharpoonup}
\nc{\nin}{\in \hs{-.35}/\,}
\nc{\np}{\newpage}
\nc{\g}{\gamma}
\nc{\IN}{I \hs{-.15} N}
\nc{\IR}{I \hs{-.14} R}
\nc{\IK}{I \hs{-.14} K}
\nc{\hs}[1]{\hspace{#1cm}}
\def\C#1{{\mathcal {#1}}}

\title[Set-valued processes and their pullback attractors]
{Minimality properties of set-valued processes and their pullback attractors}

\author[M. Coti Zelati, P. Kalita]
{Michele Coti Zelati and Piotr Kalita}

\date{\today}

\address{Indiana University - Mathematics Department
\newline\indent
Rawles Hall, Bloomington, IN 47405, USA}
\email{micotize@indiana.edu {\rm (M.\ Coti Zelati)}}

\address{Jagiellonian University - Faculty of Mathematics and Computer Science 
\newline\indent
{\L}ojasiewicza 6, 30-348 Krak\'ow, Poland}
\email{piotr.kalita@ii.uj.edu.pl {\rm (P.\ Kalita)}}



\begin{document}

\begin{abstract}
We discuss the existence of pullback attractors for multivalued dynamical systems
on metric spaces. Such attractors are shown to exist without any assumptions in 
terms of  continuity of the solution maps, based only on minimality properties with respect to
the notion of pullback attraction. When invariance is required, a very weak closed graph condition
on the solving operators is assumed. The presentation is complemented with examples and counterexamples
to test the sharpness of the hypotheses involved, including a reaction-diffusion equation, a discontinuous 
ordinary differential equation and an irregular form of the heat equation.
\end{abstract}

\maketitle

\tableofcontents


 \section{Introduction}
 \noindent The study of many evolution problems arising from mechanics and physics often 
 involves the understanding of certain classes of operators which act as solution maps to 
 systems of ordinary or partial differential equations. Given a normed space $X$, the longtime
 behavior of a differential equation of the form
\begin{equation}\label{eq:eq1}
\begin{cases}
u_t=A(t,u), \quad t>\tau\in \br,\\
u(\tau)=u_\tau\in X,
\end{cases}
\end{equation}
can be understood through the asymptotic properties of the family of solving operators
$$
U(t,\tau;\cdot):X\to X, \qquad U(t,\tau;u_\tau)=u(t),\qquad t\geq \tau\in \br.
$$
When the (possibly nonlinear) operator $A$ in \eqref{eq:eq1} does not depend explicitly on time,
the system is said to be autonomous, and the theory of such infinite-dimensional dynamical 
systems has been developed over the last four decades in many nowadays classical references \cites{dlotko, Hale,Lady,robinson,SellYou,T3}.

In the nonautonomous case, namely when the operator $A$ in \eqref{eq:eq1} is time dependent, two main objects
are widely used in the literature today. On the one hand, the concept of uniform attractor, introduced in 
\cite{CHVI94} by V.V. Chepyzhov and M.I. Vishik (see also the book \cite{CVbook}); on the other hand, the idea of pullback attractors,
started in \cites{CDF,KS,LS} for both random and nonautonomous systems.

In the above discussion, we tacitly assumed that problem \eqref{eq:eq1} is well-posed, i.e. existence and uniqueness of solutions
is available. However, for many interesting problems it is known that solutions exist globally in time, but their uniqueness is unknown
or even false in the most extreme cases. Perhaps the most famous example are the three-dimensional incompressible Navier-Stokes equations, 
studied both in the autonomous  case \cites{Ball,Ches,ChesFo,RO,SE} and the nonautonomous one \cites{CK1,CK2,ChesLu2,Voro}; other 
examples consist of
 reaction-diffusion and wave equations with non-Lipschitz nonlinearities \cites{Ball2,ChesLu1},
evolutionary differential inclusions \cites{lukaszewicz_kalita, Kasyanov2012, Kasyanov2013, melnik-1998, melnik-2000, melnik-2008, valero_book}, the surface quasi-geostrophic equation \cite{ChesDai}, 
and nonlinear Galerkin schemes \cites{CZmeas,CZTone}.

In this article, we give a fairly complete picture on the theory of nonautonomous dynamical systems and their pullback attractors. Classical
theorems on their existence are proved under minimal requirements on the set-valued maps involved: this is crucial to treat possibly ill-posed problems
like the ones mentioned above, as the lack of uniqueness is typically linked to low regularity properties of solutions. We further provide 
examples and counterexamples to investigate the sharpness of the assumptions involved in the abstract results.
 
\subsection{Plan of the article}
Section \ref{sec:multi} is dedicated to the general theory of multivalued processes, with particular emphasis on various
equivalent concepts related to asymptotic compactness. The definition of pullback attractor is given in Section \ref{sec:dissip},
in which pullback dissipativity is required to prove existence theorems on pullback attractors. At this stage, no continuity-like properties
are required on the multivalued processes, at the price of not requiring invariance of the attractors under the flow. 
Minimal continuity properties are explored in Section \ref{sec:inv}, and their link to invariance properties of limit sets is highlighted.
In Section \ref{sec:react} we provide an example of a reaction-diffusion equation with 
multivalued semilinear term having the form of the Clarke subgradient, and we prove the existence and invariance of a pullback attractor, based
on the theoretical results previously discussed. The last Section \ref{sec:point} is dedicated to certain 
examples of (autonomous) dynamical systems, some arising from (discontinuous) ODEs or PDEs, to examine
various concepts of dissipativity and invariance in highly irregular problems.

\section{Multivalued processes and  nonautonomous sets}\label{sec:multi}

\noindent Let $(X,\varrho)$ be a complete metric space, $P(X)$ the family of all nonempty  subsets of $X$,  and $\C{B}(X)$ the family of all 
nonempty and bounded subsets of $X$.
We denote by $\br_{d}=\{(s,t)\in \br^2 : s\geq t\}$. Define
$$
\dist(y,B) = \inf_{x\in B}\varrho(x,y).
$$ 
In this way, the Hausdorff semidistance between subsets of $X$ takes the form 
$$
\dist(A,B)=\sup_{y\in A}\dist(y,B).
$$ 
If $B\subset X$, then we denote $N_\eps(B) = \{x\in X\ : \ \dist(x,B)\leq \eps\}$.
We begin with a few definitions.
\begin{definition}
A mapping $U:\br_{d}\times X\to P(X)$ is called a multivalued process (\emph{m-process} for short) if:
\begin{enumerate}[label=(\roman*)]
\item $U(t,t;x)=\{x\}$ for all $t\in\br$, $x\in X$; \label{i}
\item $U(t,\tau;x)\subset U(t,s;U(s,\tau;x))$ for all $t\geq s\geq \tau$, $x\in X$, where 
$$
U(t,s;A)=\bigcup_{y\in A}U(t,s;y)
$$ 
for $A\in P(X)$ and $(t,s)\in \br_{d}$.\label{ii}
\end{enumerate}
The m-process $U$ is  called  \emph{strict}  if the equality $U(t,\tau;x) = U(t,s;U(s,\tau;x))$ holds in \ref{ii}.
\end{definition}

\begin{definition}
A family of sets $\bk=\{K(t)\subset X: t\in\br\}$ will be called a \emph{nonautonomous set}. The family $\bk$ is closed (compact, bounded) if $K(t)$ is closed (compact, bounded) for all $t\in\br$. The family 
$\bk$ is said to be \emph{backward bounded} if the sets $\bigcup_{t\leq \tau}K(t)$ are bounded for all $\tau \in\br$.
\end{definition}

\begin{definition}
The nonautonomous set $\bk$ is pullback attracting for the m-process $U$ if for every $t\in \br$ and $B\in \C{B}(X)$ we have 
$$
\lim_{\tau\to -\infty} \dist(U(t,\tau;B),K(t)) = 0.
$$
\end{definition}
 
\begin{definition}
An m-process $U$ is called \emph{pullback asymptotically compact} if for every $B\in \C{B}(X)$,  $t\in \br$ and any  sequences $\tau_n\to -\infty$
and  $\xi_n\in U(t,\tau_n;B)$, there exists $\xi\in X$ and a subsequence $\xi_{n_k}\to\xi$.
\end{definition} 
In what follows, we will omit the word ``pullback'' when referring to a pullback asymptotically compact m-process, in order to avoid
useless redundancies. It is clear that asymptotic compactness will always be understood in the pullback sense.
It is useful to compare the above definition with the asymptotic compactness of single valued processes defined in \cite{gl-2006-NATMA} where the more general formalism of cocycles is considered.

\subsection{The pullback $\omega$-limit}
An important role in our analysis is played by the so-called $\omega$-limit sets. Let $B\in \C{B}(X)$. The pullback $\omega$-limit set of
$B$ is a nonautonomous set $\Omega(B)=\{\omega(t,B): t\in\br\}$, where, for every $t\in\br$, we set
\begin{equation}\label{eq:omega}
\omega(t,B) = \bigcap_{s\leq t}\overline{\bigcup_{\tau\leq s}U(t,\tau;B)}.
\end{equation}
Equivalently, for each $t\in\br$, we can write
\begin{equation}\label{equiv:omega}
\omega(t,B) = \big\{x\in X: x_n\to x \text{ for some } x_n\in U(t,\tau_n;B), \,\tau_n\to-\infty\big\}.
\end{equation}
The importance of $\omega$-limit sets is highlighted by the following properties.
\begin{lemma}\label{lem:omegalim}
Let $U$ be an asymptotically compact m-process. Then for each $B\in \C{B}(X)$ and each $t\in \br$, $\omega(t,B)$ is
nonempty and compact. Moreover, the nonautonomous set $\Omega(B)$ pullback attracts $B$. 
\end{lemma}

\begin{proof}
The fact the for every $t\in\br$ the set $\omega(t,B)$ is nonempty follows directly from the assumption that $U$
is asymptotically compact and \eqref{equiv:omega}, while it is clear the $\omega(t,B)$ is a closed set. 

For compactness, fix $t\in\br$, and let $\{\xi_n\}_{n\in\bn}\subset \omega(t,B)$
be an arbitrary sequence.
By definition, for each $n\in\bn$ there exist sequences $\{\tau^k_n\}_{k\in\bn}$ and $\{\eta^k_n\}_{k\in\bn}\subset U(t,\tau^k_n;B)$
such that  
$$
\tau^k_n\to -\infty\quad \text{and}\quad\eta^k_n\to \xi_n\quad\text{as }k\to\infty.
$$
For each $n$ we choose $k_0(n)$ such that for $k\geq k_0(n)$ we have $\tau^k_n\leq -n$ and $\varrho(\eta^k_n,\xi_n)\leq 1/n$. 
Observe that $\eta^{k_0(n)}_n\in U(t,\tau^{k_0(n)}_n;B)$ with $\tau^{k_0(n)}_n\to -\infty$ as $n\to\infty$. Hence, from asymptotic compactness, 
for a subsequence we have $\eta^{k_0(n)}_n\to \xi$ for some $\xi\in \omega(t,B)$. 
In turn, for this subsequence we must have $\xi_n\to \xi$, which shows compactness.

Let us prove that $\omega(t,B)$ attracts $B$, namely that
$$
\lim_{\tau\to -\infty} \dist(U(t,\tau;B), \omega(t,B))=0.
$$
Suppose, for contradiction, that this is not the case. Then, we are able to find sequences 
$\tau_n\to -\infty$, $\xi_n\in U(t,\tau_n;B)$ and $\eps>0$ such that
$\dist(\xi_n, \omega(t,B)) > \eps$. Possibly passing to a subsequence, due to  asymptotic compactness
it must be that $\xi_n\to \xi$ for some $\xi\in \omega(t,B)$, which is a contradiction.
\end{proof}

The following proposition shows an equivalence
between asymptotic compactness and the attraction properties of $\omega$-limit sets.

\begin{proposition}\label{prop:asymptotic}
Let $(X,\varrho)$ be a complete metric space and let $U$ be an m-process. The following are equivalent.
\begin{itemize}
\item[(1)] $U$ is pullback asymptotically compact.
\item[(2)] For every $B\in \C{B}(X)$, the nonautonomous set $\Omega(B)$ is nonempty, compact and it pullback attracts $B$.
\end{itemize}
\end{proposition}

\begin{proof}
In light of Lemma \ref{lem:omegalim}, we only need to show one direction. Hence, assume that for every $B\in \C{B}(X)$, the nonautonomous set $\Omega(B)$ is nonempty, compact and it pullback attracts $B$. Fix $B\in \C{B}(X)$, $t\in\br$ and
sequences $\tau_n\to-\infty$ and $\xi_n\in U(t,\tau_n;B)$. We want to prove that $\{\xi_n\}_{n\in\bn}$ contains
a convergent subsequence. By assumption,  
$$
\dist(\xi_n,\omega(t,B))\leq \dist (U(t,\tau_n;B),\omega(t,B))\to 0\qquad  \text{as } n\to\infty.
$$
Therefore, by definition of Hausdorff semidistance, there exists a sequence $\eta_n\in \omega(t,B)$ such that $\varrho(\xi_n,\eta_n)\to 0$.
Since $\omega(t,B)$ is compact, we deduce the existence of a point $\xi\in \omega(t,B)$ and a subsequence 
$\eta_{n_k}\to \xi$. In turn, $\xi_{n_k}\to \xi$, and therefore  $U$ is asymptotically compact.
\end{proof}

\begin{remark}
Proposition \ref{prop:asymptotic} 
remains valid even if the metric space $(X,\varrho)$ is not complete. This generalization 
is useful in the study of doubly nonlinear equations, as shown in \cites{Seg1,RSS}.
\end{remark}

\subsection{Pullback $\omega$-limit compactness}
Given a bounded set $B\subset X$, the \emph{Kuratowski measure of noncompactness} $\kappa(B)$ of
$B$ is defined as
$$
\kappa(B)=\inf\big\{\delta\, :\, B\ \text{has a finite cover by balls of } X \text{ of diameter less than } \delta\big\}.
$$
We list hereafter some properties of $\kappa$ (see e.g. \cite{Hale}):
\begin{enumerate}[label=(K.\arabic*)]
	\item $\kappa(B)=\kappa(\overline{B})$; \label{K.1}
	\item $B_1\subset B_2$ implies that $\kappa(B_1)\leq \kappa(B_2)$; \label{K.2}
	\item $\kappa(B)=0$ if and only if $\overline{B}$ is compact; \label{K.3}
	\item fix $t\in\br$; if $\{B_s:s\leq t\}$ is a family of nonempty closed sets such that $B_{s_1}\subset B_{s_2}$ for
			$s_1<s_2$ and $\displaystyle\lim_{s\to-\infty}\kappa(B_s)=0$, 
			then $\displaystyle B=\bigcap_{s\leq t} B_s$ is nonempty and compact; \label{K.4}
	\item if $\{B_s\,:\, s\leq t\}$ and $B$ are as above, given any $s_n\to-\infty$ and any $x_n\in B_{s_n}$, there exist 
			$x\in B$ and a subsequence $x_{n_k}\to x$; \label{K.5}
	\item If $X$ is a Banach space, $\kappa(B_1+B_2)\leq \kappa(B_1)+\kappa(B_2).$ \label{K.6}
\end{enumerate}
\begin{definition}
The m-process $U$ is \emph{pullback $\omega$-limit compact} if for every $B\in \C{B}(X)$ and every $t\in\br$ we have
\begin{equation}\label{eq:kura}
\lim_{\tau\to-\infty}\kappa\left(\bigcup_{s\leq \tau}U(t,s;B)\right)=0.
\end{equation}
\end{definition}

Although somewhat obscure at first, pullback $\omega$-limit compactness is completely equivalent
to asymptotic compactness. Here, the requirement of $X$ be complete is obviously essential.

\begin{proposition}\label{prop:firstequiv}
Let $(X,\varrho)$ be a complete metric space,  and let $U$ be an m-process. The following are equivalent.
\begin{itemize}
\item[(1)] $U$ is asymptotically compact.
\item[(2)] For every $B\in \C{B}(X)$, the nonautonomous set $\Omega(B)$ is nonempty, compact and it pullback attracts $B$.
\item[(3)] $U$ is pullback $\omega$-limit compact 
\end{itemize}
\end{proposition}
\begin{proof}
First we will show that the third condition implies the second one. Choose a set $B\in \C{B}(X)$, 
$t\in \br$ and define, for $\tau\leq t$
$$B_\tau(t)=\overline{\bigcup_{s\leq \tau}U(t,s;B)}.$$
Obviously, for every $\tau\leq t$ this set is nonempty and closed. Moreover,
$B_{\tau_1}(t)\subset B_{\tau_2}(t)$ for $\tau_1\leq \tau_2\leq t$, and, by \ref{K.1} we have
$$
\lim_{\tau\to -\infty}\kappa(B_\tau(t)) = 0.
$$
Hence, from \ref{K.4} the set $\omega(t,B)$ is nonempty and compact. We show that it pullback 
attracts $B$. Suppose that this is not the case. Then there exists $\eps>0$ and  sequences 
$\tau_n\to -\infty$ and $\xi_n\in U(t,\tau_n;B)$ such that $\dist(\xi_n,\omega(t,B)) > \eps$. But, 
since $\xi_n\in B_{\tau_n}(t)$, from \ref{K.5} there must exist $\xi\in \omega(t,B)$ such that, for a 
subsequence, $\xi_n\to\xi$ and we have a contradiction.

Now we will show that the second condition implies the third one. Let $B\in\C{B}(X)$ and $t\in \br$. 
Choose $\eps>0$. We must show that there exists $\tau\in\br$ such that the set 
$\bigcup_{s\leq \tau}U(t,s;B)$ can be covered by a finite number of sets with diameter $\eps$. 
Since $\Omega(B)$ pullback attracts $B$ we know that we can find $\tau(\eps)$ such that 
for $s\leq \tau(\eps)$ we have 
\begin{equation}\label{eq:dist}
\dist(U(t,s;B),\omega(t,B))\leq \frac{\eps}{2}.
\end{equation}
Now, since $\omega(t,B)$ is compact, there exist a finite number of points $\{x_i\}_{i=1}^N$ such that 
$\omega(t,B)\subset \bigcup_{i=1}^N B\left(x_i,\eps/2\right)$. From (\ref{eq:dist}) it follows 
that if $x\in \bigcup_{s\leq \tau(\eps)}U(t,s;B)$, then $\dist(x,\omega(t,B))\leq \eps/2$. 
This means that $\varrho(x,x_i)\leq \eps$ for some $i=1,\ldots, N$ and the proof is complete.
\end{proof}

One may wonder whether condition \eqref{eq:kura} can be relaxed to 
$$
\lim_{\tau\to-\infty}\kappa\left(U(t,\tau;B)\right)=0 \qquad \forall  t\in \br,\  B\in \C{B}(X),
$$
namely, without considering the whole backward trajectory. Without any further dissipativity conditions 
(see Section \ref{sec:dissip}), this is not possible, as the following  simple example shows.

\begin{example}
Consider $X=\br$ and
$$
U(t,\tau;x) = \begin{cases}
\{x\} \ \ \mbox{for}\ \ t=\tau,\\
\{t-\tau\}\ \ \mbox{otherwise}.
\end{cases}
$$
Note that this is an m-process. Moreover, it is easy to see that
$$
\lim_{\tau\to-\infty} \kappa(U(t,\tau;B)) = 0
$$ 
for any $B\in \C{B}(\br)$ and $t\in \br$. However, for $\tau < t$  we have 
$$
\bigcup_{s\leq \tau}U(t,s;B) = (-\infty,t-\tau]
$$
and therefore none of the conditions in Proposition \ref{prop:firstequiv}  can possibly hold.
\end{example}

\subsection{The pullback flattening condition}
The notion of flattening condition was introduced in \cite{ma_2002} for semigroups and generalized in 
\cite{lukaszewicz_kalita} to set-valued semigroups, while its nonautonomous version 
(and, in fact, its nowadays used name) first appeared in \cite{KL}. It is also related to the 
concept of totally dissipative systems \cites{chep-conti-pata-2012,chep-conti-pata-2013}.

Let $X$ be a Banach space. We say that an m-process is \emph{pullback flattening} if for every $B\in \C{B}(X)$, $t\in \br$ and $\eps >0$ there exist $\tau \leq t$, a finite dimensional subspace $E$ of $X$ and a mapping $P_E:X\to E$ such that $P_E\left(\bigcup_{s\leq \tau}U(t,s;B)\right)\in \C{B}(X)$  and 
$$
(I-P_E)\left( \bigcup_{s\leq \tau}U(t,s;B) \right) \subset B(0,\eps).
$$
In Banach spaces, pullback flatness  implies pullback $\omega$-limit compactness.
\begin{lemma}\label{lem:flattening1}
If the m-process $U$ on a Banach space $X$ is pullback flattening then it is also pullback $\omega$-limit compact.
\end{lemma}
\begin{proof}
The proof follows the lines of the proof of assertion (1) in Theorem 3.10 in \cite{ma_2002} (see also the proof of Lemma 2.5 in \cite{lukaszewicz_kalita}). Let $B\in \C{B}(X)$ and $t\in \br$. We need to show that (\ref{eq:kura}) holds. Indeed, fixing $\eps > 0$, we can find certain $\tau_\eps\leq t$, a finite dimensional (and hence closed) subspace $E\subset X$, and a mapping $P_E:X\to E$ such that
$$
\bigcup_{s\leq \tau_\eps}U(t,s;B) \subset P_E\left(\bigcup_{s\leq \tau_\eps}U(t,s;B)\right) + (I-P_E)\left(\bigcup_{s\leq \tau_\eps}U(t,s;B)\right).
$$
Using \ref{K.6}, \ref{K.1}, \ref{K.2} and \ref{K.3} we have
\begin{align*}
\kappa\left(\bigcup_{s\leq \tau_\eps}U(t,s;B)\right) &\leq \kappa\left( P_E\left(\bigcup_{s\leq \tau_\eps}U(t,s;B)\right)\right) + \kappa\left((I-P_E)\left(\bigcup_{s\leq \tau_\eps}U(t,s;B)\right)\right)\\
&\leq \kappa\left(\overline{P_E\left(\bigcup_{s\leq \tau_\eps}U(t,s;B)\right)}\right) + \kappa\left(B(0,\eps)\right)\leq 2\eps,
\end{align*}
and the proof is complete.
\end{proof}
With the extra assumption of $X$ being uniformly convex, the two conditions are completely equivalent.
\begin{lemma}\label{lem:flattening2}
If the m-process $U$ on a uniformly convex Banach space $X$ is pullback $\omega$-limit compact then it is also pullback flattening.
\end{lemma}
\begin{proof}
The proof follows the lines of the proof of assertion (2) in Theorem 3.10 in \cite{ma_2002} (see also the proof of Lemma 2.6 in \cite{lukaszewicz_kalita}). Let $t\in \br$, $B\in \C{B}(X)$, and $\eps > 0$. By (\ref{eq:kura}) there exists $\tau_\eps\leq t$ such that 
$$
\kappa\left(\bigcup_{s\leq \tau_\eps}U(t,s;B)\right) < \eps,
$$
and hence
$$
\bigcup_{s\leq \tau_\eps}U(t,s;B) \subset \bigcup_{i=1}^n A_i
$$
for some sets $\{A_i\}_{i=1}^n$ of diameter less then $\eps$. Choose $x_i\in A_i$. We have 
$$
\bigcup_{s\leq \tau_\eps}U(t,s;B) \subset \bigcup_{i=1}^n B(x_i,\eps). 
$$
Denote $E=\spa\{x_1,\ldots,x_n\}$. Since $E\subset X$ is a closed and convex 
set in a uniformly convex Banach space $X$, we can define a projection operator 
$P_E:X\to E$ by 
$$
\|x-P_Ex\| = \inf_{y\in E}\|x-y\| = \dist(x,E).
$$ 
Hence for any 
$x\in \bigcup_{s\leq \tau_\eps}U(t,s;B)$ we have $\|(I-P_E)x\|\leq \eps$. 
Moreover, 
$$
\|P_Ex\|\leq \|x\| + \|(I-P_E)x\| \leq  \max_{i=1,\ldots,n}\|x_i\|+2\eps,
$$ 
and the assertion follows.
\end{proof}

Summarizing, we concentrate all the results of Proposition \ref{prop:firstequiv} and Lemmata \ref{lem:flattening1}-\ref{lem:flattening2}
in the following theorem.
\begin{theorem}\label{prop:secondequiv}
Let $X$ be a uniformly convex Banach space,  and let $U$ be an m-process. The following are equivalent.
\begin{itemize}
\item[(1)] $U$ is asymptotically compact.
\item[(2)] For every $B\in \C{B}(X)$, the nonautonomous set $\Omega(B)$ is nonempty, compact and it pullback attracts $B$.
\item[(3)] $U$ is pullback $\omega$-limit compact.
\item[(4)] $U$ is pullback flattening.
\end{itemize}
\end{theorem}

\subsection{Minimality of closed attracting sets}
Once we are able to establish the existence of a closed attracting set, it is legitimate to ask whether there exists
a smallest one. 
\begin{theorem}\label{closed_attracting}
If the m-process $U$ is asymptotically compact, then there exists 
the minimal closed nonautonomous set that pullback attracts $U$.
\end{theorem} 
\begin{proof}

For $t\in \br$ and $B\in \C{B}(X)$, from Lemma \ref{lem:omegalim} it follows that $\omega(t,B)$ is nonempty and compact, and it attracts $B$. 
Now, define the nonautonomous set $\bk=\{K(t):t\in\br\}$ by setting
$$
K(t)=\overline{\bigcup_{B\in \C{B}(X)}\omega(t,B)}.
$$ 
Obviously $K(t)$ is a nonempty closed set and $\bk$ pullback attracts $U$. 

It remains to show the minimality. Suppose that for some $t_0$ and for a closed set $C$ we have 
$$
\lim_{\tau\to -\infty} \dist(U(t_0,\tau;B), C) = 0
$$ 
for all $B\in \C{B}(X)$. We need to show that $K(t_0)\subset C$. Take $y\in K(t_0)$. 
Then, there exist sequences $B_n\in \C{B}(X)$ and $y_n\in \omega(t_0,B_n)$ such that $y_n\to y$. 
Fix $\eps>0$. For all $n>n_0$ we have $\varrho(y_n,y)\leq \eps/3$. Now since $y_n\in \omega(t_0,B_n)$ we are able to find sequences $\tau^k_n$ and $\xi^k_n\in U(t_0,\tau^k_n; B_n)$ such that $\tau^k_n \to - \infty$ and $\xi^k_n\to y_n$ as $k\to\infty$. Hence for all $n$ we can choose $k_1(n)$ such that for all $k>k_1(n)$ we have $\varrho(\xi^k_n, y_n)\leq \eps/3$. Observe that from the fact that for all $n\in\bn$ we have 
$$
\lim_{k\to\infty}\dist(U(t_0,\tau^k_n;B_n),C)=0,
$$ 
it follows that we can choose $k_2(n)$ such that for all $k>k_2(n)$ we have 
$$
\dist(\xi^k_n,C)\leq \frac{\eps}{3}.
$$ 
Now it suffices to take any $N_0>n_0$ and $K_0 > \max\{k_1(N_0),k_2(N_0)\}$ in order to  have
$$
\dist(y, C)\leq \varrho(y, y_{N_0})+ \varrho(y_{N_0}, \xi^{K_0}_{N_0}) + \dist(\xi^{K_0}_{N_0}, C) \leq \eps.
$$
Since $C$ is closed it follows that $y\in C$ and the proof is complete.
\end{proof} 

\begin{remark}
From the minimality property of the closed pullback attracting sets obtained in the above theorem, it follows that this set is
unique.
\end{remark}

\section{Dissipative m-processes and their attractors}\label{sec:dissip}
\noindent We relax the definition of pullback attractor with respect to the 
classical framework (see \cites{melnik-2000,valero_book}) by not requiring 
any invariance. Instead, we define the pullback attractor as the minimal compact 
nonautonomous set that pullback attracts all bounded sets of the phase space. 

At this stage, no continuity-type properties are needed to be postulated on the 
m-process $U$, and the existence of an attractor will follow exclusively from the
dissipativity and asymptotic features of the multivalued dynamical system in question.
The proofs are carried over from the  results for autonomous single-valued dynamical 
systems  obtained in  \cite{chep-conti-pata-2012} and, in the set-valued case, in 
\cite{Coti_Zelati_set_valued}. We mention here that, however, a careful handling 
of the non-strict case is, to the best of our knowledge, not yet present in the literature.

\subsection{Dissipativity}
While the notion of dissipativity for autonomous systems is by now well established,
the corresponding definition in the nonautonomous setting requires certain 
uniformity or monotonicity assumptions that are not immediately obvious \cites{caraballo-2010,carvalho-book,gl-2008-DCDSB}. 


\begin{definition}\label{dissipative}
An m-process $U$ is said to be \emph{dissipative} if there exists a bounded 
nonautonomous set $\bb_0=\{B_0(t): t\in\br\}$ such that for every $E\in \C{B}(X)$ 
and $t\in \br$ there exists a time $\bar{\tau}=\bar{\tau}(t,E)\leq t$ such that for all 
$\tau\leq\bar{\tau}$ we have $U(t,\tau;E)\subset B_0(t)$. It is said \emph{monotonically}
dissipative if, in addition, $B_0(s)\subset B_0(t)$ for every $s\leq t$. The nonautonomous set 
$\bb_0$ is said to be \emph{pullback absorbing}.
\end{definition}

As an immediate consequence of Theorem \ref{closed_attracting} we have the following result

\begin{theorem}\label{closed-bounded-attracting}
If the m-process $U$ is asymptotically compact and 
dissipative then there exists the minimal closed and bounded nonautonomous set that pullback attracts $U$. 
\end{theorem}

Another approach to nonautonomous problems (see for example \cite{gl-2008-DCDSB}) is to consider the so-called 
$\C{D}$-dissipative m-processes where $\C{D}$ is a family of nonautonomous sets in $X$. 
In such case we would say that an m-process is $\C{D}$-dissipative if there exists $\bb_0\in \C{D}$ such 
that for all $\be\in \C{D}$ and $t\in\br$ there exists $\bar{\tau}=\bar{\tau}(t,\be)$ such that for all 
$\tau \leq \bar{\tau}(t,\be)$ we have $U(t,\tau, E(\tau))\subset B_0(t)$.
Then both absorbing and absorbed sets must belong to the same family, 
and such approach can lead us to the uniqueness of the pullback attractor even if we do not 
require its minimality (see \cites{gl-2006-NATMA,marin-real}). Moreover, it allows to relax the 
assumptions on the nonautonomous source term in the equation (see Remark \ref{rem:source_term} in the sequel). In this paper, 
however, we consider the simpler case given by Definition \ref{dissipative} where the absorbed sets are 
bounded subsets of $X$ and the absorbing set is a bounded nonautonomous set. 
These two approaches are compared with each other for the single-valued case in \cite{marin-real}.

Before we move to the definition of a pullback attractor we provide some conditions equivalent to monotonic dissipativity.

\begin{definition} \emph{(see \cite{caraballo-2010})}
An m-process $U$ is \emph{strongly pullback bounded dissipative} if 
there exists a bounded nonautonomous set $\bb_0=\{B_0(t):  t\in\br\}$ that pullback absorbs
bounded subsets of $X$ at time $s$ for each $s\leq t$; that is, given $E\in \C{B}(X)$ and $s\leq t$ there 
exists $\bar{\tau}=\bar{\tau}(s,E)$ such that $U(s,\tau;E)\subset B_0(t)$ for all $\tau\leq \bar{\tau}$.
\end{definition}

Note that in above definition $\bar{\tau}$ may depend on $s$ and $E$ but does not depend on $t$.

\begin{proposition}\label{prop:dissip}
Let $U$ be an m-process. The following conditions are equivalent.
\begin{itemize}
\item[(1)] $U$ is monotonically dissipative.
\item[(2)] $U$ is dissipative and the nonautonomous absorbing set is backward bounded.
\item[(3)] $U$ is strongly pullback bounded dissipative.
\end{itemize}
\end{proposition}
\begin{proof}
We divide the proof in several steps.

\medskip

\noindent (1) $\Rightarrow$ (2). From the monotonicity of the absorbing set $\bb_0$ it 
immediately follows that 
$$
\bigcup_{\tau\leq t} B_0(\tau) = B_0(t)
$$ 
and hence $\bb_0$ is backward bounded.

\medskip

\noindent (2) $\Rightarrow$ (1). From the fact that the absorbing set $\bb_0$ is backward bounded 
it follows that  $\mathbb{C}=\left\{\bigcup_{\tau\leq t} B_0(\tau)\ :\ t\in\mathbb{R}\right\}$ is bounded, absorbing and monotone.

\medskip

\noindent (1) $\Rightarrow$ (3). From the monotonicity of the absorbing set $\bb_0$ it immediately follows that for all $\tau\leq \bar{\tau}(s,E)$ we have  $U(s,\tau;E)\subset B_0(s)\subset B_0(t)$ for all $t\geq s$.

\medskip

\noindent (3) $\Rightarrow$ (1). Since 
$$
\bigcup_{\tau\leq \bar{\tau}(s,E)}U(s,\tau;E)\subset B_0(t)
$$ 
for all $t\geq s$, then  
$$
\bigcup_{\tau\leq \bar{\tau}(s,E)}U(s,\tau;E)\subset \bigcap_{t\geq s}B_0(t).
$$ 
The family $\mathbb{C}=\left\{ \bigcap_{t\geq s}B_0(t)\ :\ s\in\mathbb{R}\right\}$ is bounded, absorbing and monotone.
\end{proof}

\subsection{Pullback attractors: a first result}

A nonautonomous set $\ba=\{A(t): t\in \br\}$ is a \emph{pullback attractor} for the m-process 
$U$ if it is compact, pullback attracting and minimal in the class of closed pullback attracting 
nonautonomous sets. Its existence is characterized by the following theorem.

\begin{theorem}\label{thm:eqivalence_attractor}
Let $U$ be an m-process. $U$ is pullback asymptotically compact and monotonically dissipative 
if and only if it possesses a unique backward bounded pullback attractor.
\end{theorem}
\begin{proof}
First we prove the sufficiency. Let $\bb_0=\{B_0(t): t\in\br\}$ be a backward bounded pullback absorbing nonautonomous set. 
In view of Lemma \ref{lem:omegalim},
for each $t\in\br$ the sets $A(t)=\omega(t,B_0(t))$ are nonempty, compact and 
\begin{equation}\label{eq:disttoo}
\lim_{s\to-\infty} \dist(U(t,s; B_0(t)), A(t))=0, \qquad \forall t\in\br.
\end{equation}
We need to show that $\ba=\{A(t): t\in\br\}$ pullback attracts any bounded set of the phase space. If $E\in \C{B}(X)$ is such a set,
then for each $s,t\in\br$ with $s\leq t$, we can find $\bar\tau=\bar\tau(s,E)< s$ such that 
$$
U(t,\tau;E)\subset U(t,s; U(s,\tau;E))\subset U(t,s; B_0(s))\subset U(t,s;B_0(t)), \qquad \forall \tau\leq \bar\tau,
$$
where the last inclusion follows from monotonicity. Therefore, for every $t\in\br$ and $\tau<\bar\tau$ 
$$
\dist(U(t,\tau;E),A(t))\leq \dist(U(t,s;B_0(t)),A(t)).
$$
Hence,
$$
\limsup_{\tau\to-\infty}\dist(U(t,\tau;E),A(t))\leq \dist(U(t,s;B_0(t)),A(t)),
$$
and, in light of \eqref{eq:disttoo}, the pullback attraction property follows by taking the limit as $s\to-\infty$ in the above
inequality.
For minimality, observe that Theorem \ref{closed_attracting} provides the existence of
a uniquely determined minimal closed pullback attracting nonautonomous set, defined as
$$
K(t) = \overline{\bigcup_{B_0\in \C{B}(X)}\omega(t,B_0)}.
$$
Since $B_0(t)$ is bounded, for all $t\in\br$ we have
$$
A(t)=\omega(t,B_0(t))\subset K(t) \subset A(t)\qquad \Rightarrow \qquad A(t)=K(t),
$$
proving minimality.

To show that $\ba$ is backward bounded observe that 
for all $t\in\br$ there exists $\bar{\tau}(t,B_0(t))$ such that we have 
$\bigcup_{\tau\leq \bar{\tau}(t,B_0(t))}U(t,\tau,B_0(t))\subset B_0(t)$ and 
hence $\overline{\bigcup_{\tau\leq \bar{\tau}(t,B_0(t))}U(t,\tau,B_0(t))}\subset \overline{B_0(t)}$. 
This means that $A(t)=\omega(t,B_0(t))\subset \overline{B_0(t)}$. Hence we have 
$$
\bigcup_{\tau\leq t}A(t)\subset \bigcup_{\tau \leq t}\overline{B_0(t)}\subset \overline{\bigcup_{\tau\leq t}B_0(t)}=\overline{B_0(t)}
$$ 
and the backward boundedness follows.

We move to the proof of necessity. Let $E\in \C{B}(X)$. We have for all $t\in \br$ that
$$\lim_{s\to -\infty}\dist(U(t,s;E),A(t)) = 0.$$ It follows that
$$
\lim_{s\to -\infty}\dist\left(U(t,s;E),\bigcup_{\tau\leq t}A(\tau)\right) = 0.
$$
Choose $\eps>0$. The nonautonomous set $\mathbb{C} = \left\{N_\eps\left(\bigcup_{\tau\leq t}A(\tau)\right)\ :\ t\in\br\right\}$ is 
bounded and monotone. We will prove that it must be absorbing. Assume for the sake of contradiction that there exists a set 
$E\in \C{B}(X)$ and sequences $\{x_n\}\subset E$, $s_n\to -\infty$ and $\xi_n\in U(t,s_n;x_n)$ such that 
$\xi_n\notin N_\eps\left(\bigcup_{\tau\leq t}A(\tau)\right)$. This means that $\dist\left(\xi_n, \bigcup_{\tau\leq t}A(\tau)\right) > \eps$ 
and $\dist\left(U(t,s_n;E), \bigcup_{\tau\leq t}A(\tau)\right) > \eps$, a contradiction. For the proof of asymptotic compactness 
assume that $\{x_n\}\subset E$ where $E\in\C{B}(X)$ and $\xi_n\in U(t,s_n;x_n)$ with $s_n\to -\infty$. We have 
$$
\lim_{n\to\infty}\dist(\xi_n,A(t))=0,
$$ 
and hence we are able to find a sequence $\{\eta_n\}\subset A(t)$ such that $\varrho(\xi_n,\eta_n)\to 0$ as $n\to \infty$. From the fact that $A(t)$ is compact it follows that, for a subsequence denoted by the same index, $\eta_n\to \eta$. For this subsequence we must have $\xi_n\to \eta$ and the proof is complete.  
\end{proof}

\begin{remark}
The above theorem for the case of continuous single-valued processes was proved in \cite{gl-2008-DCDSB}. We stress one
more time that, up to now, nothing is assumed as far as continuity of the m-process.
\end{remark}

Note that in a finite dimensional setting we do not need to assume the monotonicity of the absorbing sets.

\begin{theorem}
Let $X$ be a finite dimensional Banach space. An m-process $U$ is pullback asymptotically compact and dissipative if and only if it possesses a unique pullback attractor. 
\end{theorem}

\begin{proof}
Sufficiency follows immediately from the Theorem \ref{closed-bounded-attracting}. 
Necessity follows the lines of the proof of necessity in Theorem \ref{thm:eqivalence_attractor}. 
The proof of asymptotic compactness is exactly the same as the corresponding proof in Theorem 
\ref{thm:eqivalence_attractor}. For the proof of dissipativity observe that if we define the 
nonautonomous set $\mathbb{C}=\{N_\eps(A(t))\ : t \in \br\}$ then it follows that it is bounded 
and absorbing and the proof is complete.  
\end{proof}

\begin{remark}
If we define the pullback attractor as a nonautonomous set that is compact, attracting and 
invariant then this set does not have to be defined uniquely (see \cites{caraballo-2010,carvalho-book,MZ}). 
Here we impose the minimality rather then the invariance in the definition of the pullback attractor and hence 
the attractor, if it exists, it must be unique even if it is backward unbounded.
\end{remark}

\subsection{Equivalent conditions}
We summarize in this paragraph the main results on the existence of pullback attractors. As we have seen above,
it is necessary to distinguish between the case of a general metric space and a (uniformly convex) Banach space.
On account of Proposition \ref{prop:dissip}, we state our result for monotonically dissipative processes, having in 
mind that other equivalent notions of dissipativity can be assumed. We begin with the case in which the phase space
$X$ is assumed to be a complete metric space.

\begin{theorem}\label{thm:metricequiv}
Let $(X,\varrho)$ be a complete metric space,  and let $U$ be a monotonically dissipative 
m-process acting on $X$, with bounded nonautonomous absorbing set 
$\bb_0=\{B_0(t): t\in\br\}$. The following are equivalent.
\begin{enumerate}

\item $U$ is asymptotically compact.

\item For every $t\in\br$, $\omega(t,B_0(t))$ is nonempty, compact and attracts $B_0(t)$.

\item $U$ is pullback $\omega$-limit compact.

\item There exists the backward bounded pullback attractor of $U$.

\end{enumerate}
\end{theorem}

\begin{proof}
Proposition \ref{prop:firstequiv} and Theorem \ref{thm:eqivalence_attractor} ensure the equivalence among (1), (3) and
(4). Clearly, a further use of  Proposition \ref{prop:firstequiv} also ensures
that (1) implies (2). Let us now prove that (2) implies (3).  Fix a bounded set $B\subset X$. By definition of dissipativity, 
for every $s\in\br$, there exists $\bar\tau=\bar\tau(s,B)\leq s$ such that
$$
U(s,\tau;B)\subset B_0(s), \qquad \forall \tau \leq \bar\tau.
$$
Let now $\tau\leq s\leq t$. From the definition of m-process and the above dissipativity condition, we have
$$
U(t,\tau;B)\subset U(t,s;U(s,\tau;B))\subset U(t,s;B_0(s)), \qquad\forall \tau\leq \bar\tau\leq s\leq t.
$$
Appealing to the monotonicity of the absorbing set, the above inclusion implies
$$
U(t,\tau;B)\subset  U(t,s;B_0(t)), \qquad\forall \tau\leq \bar\tau\leq s\leq t.
$$
As a consequence
$$
U(t,\tau;B)\subset \bigcup_{r\leq s} U(t,r;B_0(t)), \qquad\forall \tau\leq \bar\tau\leq s\leq t,
$$
and, therefore,
$$
\bigcup_{r\leq \tau} U(t,r;B)\subset \bigcup_{r\leq s} U(t,r;B_0(t)), \qquad\forall \tau\leq \bar\tau\leq s\leq t.
$$
Property \ref{K.2} of the Kuratowski measure then implies that
$$
\kappa\left(\bigcup_{r\leq \tau} U(t,r;B)\right)\leq \kappa\left(\bigcup_{r\leq s} U(t,r;B_0(t))\right), \qquad\forall \tau\leq \bar\tau\leq s\leq t.
$$
Arguing as in Proposition \ref{prop:firstequiv}, it is not hard to see that
$$
\lim_{s\to-\infty}\kappa\left(\bigcup_{r\leq s} U(t,r;B_0(t))\right)=0,
$$
from which we infer that
$$
\lim_{\tau\to-\infty}\kappa\left(\bigcup_{r\leq \tau} U(t,r;B)\right)=0.
$$
Hence, $U$ is pullback $\omega$-limit compact, and the proof is concluded.
\end{proof}

In uniformly convex Banach spaces, also the pullback flattening condition can be added to the above list.
\begin{theorem}\label{thm:metricequiv2}
Let $X$ be a uniformly convex Banach space,  and let $U$ be a monotonically dissipative 
m-process acting on $X$, with bounded nonautonomous absorbing set 
$\bb_0=\{B_0(t): t\in\br\}$. The following are equivalent.
\begin{enumerate}

\item $U$ is asymptotically compact.

\item For every $t\in\br$ the set $\omega(t,B_0(t))$ is nonempty, compact and attracts $B_0(t)$.

\item $U$ is pullback $\omega$-limit compact.

\item $U$ is pullback flattening.

\item There exists the backward bounded pullback attractor of $U$.

\end{enumerate}
\end{theorem}
The proof of the above theorem follows immediately from Theorems  
\ref{prop:secondequiv} and \ref{thm:metricequiv}.

\section{Invariance}\label{sec:inv}
\noindent We discuss here the invariance properties of the pullback 
attractor in both the strict and the non-strict case. In this section,
the concepts of $t_\star$-closed and closed processes play an essential role and the phase
space will be a complete metric space $(X,\varrho)$ on which the m-process $U$ acts.
Specifically, we will discuss sufficient conditions under which the pullback attractor
$\ba=\{A(t):t\in\br\}$ is either
\begin{itemize}
	\item \emph{negatively invariant}: $A(t)\subset U(t,s)A(s)$ for every $s\leq t\in \br$;
	\item \emph{invariant}:  $A(t)= U(t,s)A(s)$ for every $s\leq t\in \br$.
\end{itemize}
Generically speaking, non-strict m-processes possess negatively invariant attractors, while
strict processes have (fully) invariant ones.

\subsection{Weak continuity properties of m-processes}
The generalization of the concept of continuity from the case of a  single-valued function to that of a set-valued
one can be obtained by using the notions of upper and lower semicontinuity. When dealing with multivalued dynamical
systems, a natural notion to consider is upper semicontinuity, which usually needs to be complemented with the assumption that the values
$U(t,\tau;\xi)$ are closed for every $\tau\leq t\in\br$ and every $\xi\in X$. In turn, this implies that the graph of $U(t,\tau;\cdot)$ is
closed--a much weaker condition than continuity in the single-valued case. In the context of single-valued dynamical
systems, the closed graph condition was first used 
in \cite{PZ}, and then generalized to even weaker notions (see below)
in \cite{chep-conti-pata-2012} for semigroups and in \cites{chep-conti-pata-2013, CPT2013} for processes. 
For the multivalued case, the reader is referred to \cites{melnik-1998,melnik-2000, Coti_Zelati_set_valued, lukaszewicz_kalita, KKV}.

\begin{definition}
Let $U$ be an m-process acting on a complete metric space $(X,\varrho)$. Then $U$ is said to be 
\begin{itemize}
\item \emph{$t_\star$-closed} if there exists $t_\star>0$ such that $U(t,t-t_\star;\cdot)$ has closed graph
for every $t\in\br$, namely if the following implication holds true: 
$$
\eta_n\to \eta,\quad U(t,t-t_\star;\eta_n) \ni \xi_n\to \xi \qquad \implies \qquad \xi\in U(t,t-t_\star;\eta);
$$
\item \emph{closed} if it is $t_\star$-closed for every $t_\star>0$.
\end{itemize}
\end{definition}

\subsection{The non-strict case} The negative invariance of the pullback attractor is strictly related to the points at which the graph
of the m-process is closed.
 
\begin{proposition}\label{prop:nonstrictinv}
If $U$ is a $t_\star$-closed, asymptotically compact and monotonically dissipative m-process, then the pullback attractor $\ba=\{A(t):t\in\br\}$ satisfies
the inclusion
\begin{equation}\label{eq:tneginv}
A(t)\subset U(t,t-t_\star; A(t-t_\star)), \qquad \forall t\in\br.
\end{equation}
If, furthermore, $U$ is a closed m-process, then the attractor is negatively invariant, namely
\begin{equation}\label{eq:neginv}
A(t)\subset U(t,s; A(s)), \qquad \forall s\leq t\in\br.
\end{equation}
\end{proposition}

\begin{proof}
It is clear that all we need to prove is \eqref{eq:tneginv}, as the second assertion will follow just by
the possibility to take $t_\star=t-s$, obtaining (\ref{eq:neginv}). According to Theorem \ref{thm:metricequiv},
there exists a bounded nonautonomous absorbing set $\bb_0=\{B_0(t): t\in\br\}$ such that
$$
A(t)=\omega(t,B_0(t)), \qquad \forall t\in\br.
$$
Let $\xi\in A(t)$. Then, there exist sequences $\tau_n\to -\infty$ and  $\xi_n \in U(t,\tau_n; B_0(t))$ such that
$\xi_n\to \xi$ as $n\to \infty$. We deduce that (for $n$ big enough)
\begin{align*}
\xi_n\in U(t,\tau_n; B_0(t))&\subset U(t,t-t_\star; U(t-t_\star, \tau_n; B_0(t))).
\end{align*}
Hence, $ \xi_n\in U(t,t-t_\star; \eta_n)$ for some $\eta_n \in U(t-t_\star, \tau_n; B_0(t))$. By asymptotic compactness,
there exists a (non-relabeled) subsequence such that 
$$
\eta_n\to \eta\in \omega(t-t_\star,B_0(t))\subset A(t-t_\star),
$$
where the last inclusion follows from the fact that $\mathbb{A}$ is pullback attracting. Since the graph of $U(t,t-t_\star;\cdot)$ is closed, we infer that
$$
 \eta_n\to \eta, \qquad U(t,t-t_\star; \eta_n) \ni \xi_n\to \xi \qquad \implies \qquad  \xi \in U(t,t-t_\star; \eta).
$$
Hence,
$$
\xi\in U (t,t-t_\star; A(t-t_\star))\qquad \implies \qquad A(t)\subset  U(t,t-t_\star; A(t-t_\star)),
$$
concluding the proof.

\end{proof}

\subsection{The strict case}
Surprisingly, the picture in this case is quite different, and the strictness assumption entails
a much stronger result under much weaker assumptions. Indeed, we shall first prove that \eqref{eq:tneginv}
is enough to ensure the full invariance of the pullback attractor.

\begin{proposition}\label{prop:strictinv}
Let $U$ be a strict m-process possessing a backward bounded pullback attractor $\ba=\{A(t):t\in\br\}$. If there exists $t_\star\in \br$ such
that  
\begin{equation}\label{eq:suffinv}
A(t)\subset U(t,t-t_\star; A(t-t_\star)), \qquad \forall t\in\br,
\end{equation}
then, in fact,
\begin{equation}\label{eq:inv0}
A(t)= U(t,s) A(s), \qquad \forall s\leq t\in\br.
\end{equation}
\end{proposition}

\begin{proof}
Let $t\in\br$ be fixed. Arguing by induction and using the strictness of the m-process, 
for any $\tau\geq t$ and any integer $n$ we obtain the inclusion
\begin{equation}\label{eq:ind}
\begin{aligned}
U(\tau,t;A(t))&\subset U(\tau,t; U(t,t-t_\star; A(t-t_\star)))=U(\tau,t-t_\star; A(t-t_\star))\\
&\subset U(\tau,t-2t_\star; A(t-2t_\star))\ldots \subset U(\tau,t-nt_\star; A(t-nt_\star)).
\end{aligned}
\end{equation}
Since $\ba$ is backward bounded and pullback attracting , we obtain
$$
\begin{aligned}
\dist (U(\tau,t;A(t)), A(\tau))&\leq \dist (U(\tau,t-nt_\star; A(t-nt_\star)), A(\tau))\\
&\leq\dist \left(U\left(\tau,t-nt_\star; \bigcup_{s\leq t}A(s)\right), A(\tau)\right)\to 0,
\end{aligned}
$$
as $n\to \infty$. Thus 
\begin{equation}\label{eq:ind2}
U(\tau,t;A(t))\subset A(\tau), \qquad \forall \tau\geq t,
\end{equation}
since $A(\tau)$ is closed. On the other hand, setting $\tau=t$ in \eqref{eq:ind} and using \eqref{eq:ind2}, we have that
$$
A(t)\subset U(t,t-nt_\star; A(t-nt_\star))\subset A(t), 
$$
namely
\begin{equation}\label{eq:ind3}
A(t)= U(t,t-nt_\star; A(t-nt_\star)), \qquad \forall n\in\bn. 
\end{equation}
Now, fix $s\leq t$ and take $n$ large enough so that $t-nt_\star \leq s$. Using \eqref{eq:ind2}-\eqref{eq:ind3}, we have
\begin{equation}\label{eq:ind4}
\begin{aligned}
A(t)&=U(t,t-nt_\star; A(t-nt_\star))\subset U(t,s;U(s,t-nt_\star; A(t-nt_\star)))\\
&\subset U(t,s;A(s))\subset A(t),
\end{aligned}
\end{equation}
proving the reverse of the inclusion \eqref{eq:ind2}, and hence the sought equality.
\end{proof}
Notice that in \eqref{eq:ind4}, the second inclusion is in fact an equality due to the strictness of the m-process. 
We stress here that strictness was used in a crucial way only in the induction procedure in \eqref{eq:ind}.

Combining Propositions \ref{prop:nonstrictinv} and \ref{prop:strictinv}, it is immediate to deduce full invariance
of the pullback attractor for strict m-processes.
\begin{corollary}\label{corollary:inv}
If $U$ is a $t_\star$-closed, asymptotically compact and monotonically dissipative strict m-process, 
then the pullback attractor $\ba=\{A(t):t\in\br\}$ is invariant, namely
\begin{equation}\label{eq:inv}
A(t)= U(t,s) A(s), \qquad \forall s\leq t\in\br.
\end{equation}
\end{corollary}
\begin{remark}
The above theorem is a generalization of Theorem 1 from \cite{KKV}. Indeed, the closed graph 
requirement  is replaced here by the weaker $t_\star$-closed property.
\end{remark}

\section{A reaction-diffusion equation with multivalued semilinear term}\label{sec:react}

\noindent This section presents an example of an initial and 
boundary value problem governed by a reaction-diffusion equation with 
multivalued semilinear term having the form of the Clarke subgradient. Differential inclusions with multivalued terms 
in the form of the Clarke subdifferentials were introduced by Naniewicz and Panagiotpoulos \cite{Naniewicz} to express 
the friction and normal contact laws in the theory of elasticity (see also \cite{migorski-book} for a more up-to-date overview of the subject).
Since then, such problems were also used to model the diffusion through semipermeable membranes \cite{miettinen} and boundary feedback 
control laws for heat conduction problems \cite{Wang}. The formalism of the Clarke subdifferential can be used to put in a common 
abstract framework three models recalled in \cite{balibrea}*{Sections 4.1.3-4.1.4}: a model of combustion in porous media, a model 
of conduction of electrical impulses in nerve axons and a model from climatology.
After proving a global existence result for weak solutions, 
we show that the generated multivalued (strict) process possesses a fully invariant pullback attractor.

\subsection{Problem formulation} 
Recall that for a locally Lipschitz functional $J$ on a  Banach space $X$,
the generalized directional derivative in the sense of Clarke at the point $x\in X$ and direction $y\in X$ has the form
$$
J^0(x;y)=\limsup_{z\to x, \lambda\to 0^+}\frac{J(z+\lambda y)-J(z)}{\lambda},
$$
while the Clarke subgradient of $J$ at the point $x\in X$ is given as
$$
\partial J(x)=\{\xi\in X^*\,:\,\langle \xi,y\rangle \leq J^0(x;y)\ \mbox{for all}\ y\in X \}.
$$
If $X=\br^n$ then a following simple characterization holds
$$
\partial J(x) = \conv\left\{\lim_{n\to \infty}J'(x_n)\,:\, x_n\to x,\ x_n\notin S,\  \mbox{and}\  \{J'(x_n)\}\  \mbox{converges}\right\},
$$
where $S$ is any Lebesgue null set containing the nondifferentiability points of $J$ in the neighborhood of $x$. 
For more details and properties of the Clarke subgradient see \cites{clarke, Naniewicz, denkowski, migorski-book}. 
In the sequel, for functionals of more than one variable we will always understand the symbol 
$\partial$ as the Clarke subdifferential with respect to the last variable. Let $\Omega\subset \br^d$, 
$d\in \bn$ be a bounded open domain with Lipschitz boundary $\partial \Omega = \overline{\Gamma}_D\cup\overline{\Gamma}_N$, 
where $\Gamma_D$ and $\Gamma_N$ are mutually disjoint, relatively open and nonempty. 
Let $t_0\in \br$ be an initial time. Consider the following initial and boundary value problem, where 
we are looking for $u:\overline{\Omega}\times [t_0,\infty)\to \br$ such that
\begin{align}
  u_t(x,t)-\Delta u(x,t)+ \partial j(x,t,u(x,t)) &\ni f_0 (x,t) \ & &\mbox{on}\  \Omega\times (t_0,\infty),\label{eq:inclusion}\\
  u(x,t) &=  0 \  & &\mbox{on} \ \Gamma_D\times (t_0,\infty),\\
  \frac{\partial u(x,t)}{\partial \nu}  &=  f_N(x,t)\  & &\mbox{on}\ \Gamma_N\times (t_0,\infty),\\
  u(x,t_0)  &=  u_{0}(x)\  & &\mbox{on}\ \Omega.
\end{align}
Define $V=\{v\in H^1(\Omega)\, :\, v= 0\ \mbox{on}\ \Gamma_D\}$ and $H=L^2(\Omega)$. 
Then $V\subset H\subset V^*$ constitute an evolution triple with the embeddings being dense, 
continuous, and compact. The norm in $V$ and the duality pairing between $V$ and $V^*$ will be denoted 
by $\|\cdot\|$ and $\langle\cdot,\cdot\rangle$, respectively. For all other spaces we will use appropriate 
subscripts to denote the corresponding norms. Note that we have the Poincar\'{e} inequality 
$\lambda_1\|v\|_H^2\leq \|v\|^2$ for $v\in V$, where $\lambda_1$ is the first eigenvalue of the Laplace 
operator. The scalar product in $H$ will be denoted by $(\cdot,\cdot)$. Moreover, for a multifunction 
$\xi:\Omega\to 2^\br$ we will use the notation $S^2_\xi$ for all its selections that belong to $H$. Define $f(t)\in V^*$ as
$$
\langle f(t),v\rangle=\int_{\Omega} f_0(x,t)v(x)\, \d x+ \int_{\Gamma_N}f_N(x,t)v(x)\, \d\Gamma\qquad\mbox{for}\ v\in V,
$$
and the mapping $A:V\to V^*$ as
$$
\langle Au,v\rangle = \int_{\Omega}\nabla u(x)\cdot \nabla v(x)\, \d x.
$$
Then $\langle Av,v\rangle = \|v\|^2$ for $v\in V$. We consider the following weak formulation of the above problem.

\medskip

\noindent \textbf{Problem $(\C{P})$}. Find $u\in L^2_{loc}(t_0,\infty;V)$ with $u'\in L^2_{loc}(t_0,\infty;V^*)$ and $\eta\in L^2_{loc}(t_0,\infty;H)$ such that
\begin{align}
&\langle u'(t),v\rangle + \langle A u(t), v\rangle + (\eta(t),v) = \langle f(t),v\rangle \ \ \mbox{for all}\ v\in V\ \mbox{a.e.}\ t\in (t_0,\infty)\label{eq:weak_form}\\
& u(t_0) = u_0,\\
& \eta(t) \in S^2_{\partial j(\cdot,t,u(\cdot,t))}\ \ \mbox{a.e.}\ t\in (t_0,\infty)\label{eq:selection}.
\end{align}
\begin{remark} In place of the Laplace operator in Problem $(\C{P})$ it is possible to consider a linear, 
continuous, positively defined and symmetric operator. In the sequel of the proof it will be essential that 
this operator has an orthonormal in $H$ and orthogonal in $V$ basis of eigenfunctions such that the 
corresponding eigenvalues $0 < \lambda_1 \leq \lambda_2\leq\ldots\leq\lambda_n\leq\ldots$ go to infinity as $n\to \infty$.
\end{remark}
\begin{remark}
Consider the equation
\begin{equation}\label{eq:discontinuous}
 u_t(x,t)-\Delta u(x,t) + h(x,t,u(x,t))=f_0(x,t),
 \end{equation}
 where $h(x,t,\cdot)\in L^\infty_{loc}(\br)$ can be discontinuous. Then we can define the functional 
$j(x,t,u)=\int_0^u h(x,t,s)\, \d s$. This functional is locally Lipschitz with respect to the variable $u\in \br$ and we can 
consider its Clarke subdifferential with respect to this variable. Inclusion (\ref{eq:inclusion}) can 
be understood as a multivalued regularization of (\ref{eq:discontinuous}) that guarantees the existence 
of weak solutions assuming that the growth condition \ref{J2} stated below holds.
\end{remark}

\subsection{Assumptions and existence of weak solutions} We make the following assumptions on the problem data.

\begin{enumerate}[label=(J\arabic*)]
	\item \label{J1} The functional $j:\Omega\times \br\times\br\to \br$ is such that for all $s\in\br$ 
	the mapping $(x,t)\to j(x,t,s)$ is measurable and for almost all $(x,t)\in \Omega\times \br$ 
	the mapping $s\to j(x,t,s)$ is locally Lipschitz. 
	
	\item \label{J2} The following growth condition holds:  for almost all $(x,t)\in \Omega\times \br$ and for all 
	$s\in \br$, $\xi\in \partial j(x,t,s)$ we have $|\xi|\leq c_1+c_2|s|$ with constants $c_1,c_2>0$.
	
	\item \label{J3} The following dissipativity condition holds: for almost all $(x,t)\in \Omega\times \br$ and 
	for all $s\in \br$, $\xi\in \partial j(x,t,s)$ we have $\xi s\geq d_1-d_2|s|^2$ with constants $d_1\in \br$ 
	and $d_2\in [0,\lambda_1)$.

\end{enumerate}
	
\begin{enumerate}[label=(F\arabic*)]	
	\item\label{F1} $f\in L^2_{loc}(\br;V^*)$.
	
	\item\label{F2} For some $\bar{t}\in\br$ we have $\|f\|_{L^\infty(-\infty,\bar{t};V^*)} < \infty$.
	
\end{enumerate}

\begin{remark}\label{rem:source_term}
Assumption \ref{F2} is stronger than the assumption that for all $t\in \br$ we have
$$
\int_{-\infty}^{t}\e^{\sigma s}\|f(s)\|_{V^*}^2\, \d s<\infty
$$ 
for certain $\sigma>0$ dependent on the problem data, used for example in \cites{gl-2006-CR, gl-2006-NATMA, GMRR}. This weaker assumption 
would lead to the notion of pullback $\C{D}$-attractor, where $\C{D}$ is an appropriately constructed family 
of bounded nonautonomous sets. See \cite{marin-real} for the comparison of the two notions of pullback attractors. 
Hypothesis \ref{F2} was used for example in \cite{langa-lukaszewicz-real}, where  estimates on the fractal dimension 
of the pullback attractor for two dimensional Navier-Stokes flows in unbounded domains were obtained. 
Moreover, \ref{F2} can be relaxed to the weaker assumption that for certain $\bar{t}\in \mathbb{R}$ we have 
$$
\sup_{t\leq \bar{t}} \e^{-\sigma t}\int_{-\infty}^{t}\e^{\sigma s}\|f(s)\|_{V^*}\, \d s < \infty,
$$ 
where $\sigma$ is a certain constant depending on the problem data (see \cite{marin-real} and proof of Lemma 
\ref{lemma:dissipative_a_priori} below, which is the only place where this assumption is used).
\end{remark}
\begin{theorem}
Let $t_0\in\br$ and $u_{0}\in H$. Under assumptions \ref{J1}-\ref{J2} and \ref{F1}, Problem $(\C{P})$ has a (possibly non-unique) solution.
\end{theorem} 
\begin{proof}
We will prove the existence of solution $u$ on the interval $(t_0,t_0+1)$. This solution will be proved to belong to the 
space $\{v\in L^2(t_0,t_0+1;V)\, :\, v'\in L^2(t_0,t_0+1;V^*)\}$. Since this space embeds in $C([t_0,t_0+1];H)$, we can 
take the solution's endpoint $u(t_0+1)$ and use it as an initial condition for a new trajectory on the interval $(t_0+1,t_0+2)$. 
The concatenation of these two trajectories will be a solution on $(t_0,t_0+2)$. Proceeding recursively, we obtain the 
solution on the whole interval $(t_0,\infty)$. The proof of existence on $(t_0,t_0+1)$ is based on the arguments of \cite{miettinen} 
and hence it will be sketched only briefly here. Let $\varrho\in C^\infty_0(\br)$ be a mollifier kernel such that 
$\supp\, \varrho\subset [-1,1]$, $\varrho(s)\geq 0$ for all $s\in \br$ and $\int_{\br}\varrho(s)\, \d s=1$. Define $\varrho_n(s)=n\varrho(ns)$ 
and put
$$
j_n(x,t,s)=\int_{\br}\varrho_n(\eta) j(x,t,s-\eta)\, \d\eta.
$$
By $j_n'(x,t,s)$ we will denote the derivative of $j_n$ with respect to the variable $s$. Note that $j_n'$ satisfies 
the growth condition \ref{J2} with constants $c_1,c_2$ different than those for $\partial j$ but independent 
on $n$. Let $V_n=\spa\{z_1,\ldots,z_n\}$, where $\{z_n\}_{n=1}^\infty$ are orthonormal in $H$ and orthogonal 
in $V$ eigenfunctions of the operator $A$ corresponding to the eigenvalues $0<\lambda_1\leq \lambda_2\leq \ldots \leq \lambda_n \leq \ldots$. 
The projection operator on $V_n$ will be denoted by $P_n$. Note that we have $\|P_n v\|_{X}\leq \|v\|_X$ and 
$P_n v\to v$ in $X$ for all $v\in X$, where $X=V, H,$ or $V^*$ (see Lemma 7.5 in \cite{robinson}). We formulate 
the following regularized Galerkin problem

\medskip

\noindent \textbf{Problem $(\C{P})_n$}. Find $u_n\in AC(t_0,t_0+1;V_n)$ such that $u_n(t_0)=P_nu_0$ and 
for almost every $t\in (t_0,t_0+1)$ and all $v\in V_n$ we have
\begin{equation}\label{eq:galerkin_problem}
\langle u_n'(t) + A u_n(t),v\rangle + \int_{\Omega}j_n'(x,t,u_n(x,t))v(x)\, \d x = \langle P_nf(t), v\rangle.
\end{equation}
If $u_n$ solves Problem $(\C{P})_n$, then, taking $v=u_n(t)$ and using \ref{J2} we have for a.e. $t\in (t_0,t_0+1)$

\begin{align*}
\frac{1}{2}\frac{\d}{\d t}\|u_n(t)\|_H^2 + \|u_n(t)\|^2 &\leq \|f(t)\|_{V^*}\|u_n(t)\| + \int_{\Omega}c_1|u_n(x,t)|+c_2|u_n(x,t)|^2\,\d x \\
& \leq \frac{\|u_n(t)\|^2}{2}+\frac{\|f(t)\|_{V^*}^2}{2} + \left(c_2+\frac{1}{2}\right) \|u_n(t)\|_H^2 + \frac{c_1^2}{2}m(\Omega), 
\end{align*}
whence, for all $t\in [t_0,t_0+1]$ we get
$$
\|u_n(t)\|_H^2 \leq \|u_0\|_H^2 + \|f\|_{L^2(t_0,t_0+1;V^*)}^2 + c_1^2 m(\Omega) + (2c_2+1)\int_{t_0}^t \|u_n(\tau)\|_H^2\, \d \tau.
$$
As a consequence, from the Gronwall lemma, we get the bound
\begin{equation}\label{eq:Galerkin_bound_1}
\|u_n\|_{L^2(t_0;t_0+1;V)}+\|u_n\|_{C([t_0,t_0+1];H)}\leq C(\|u_0\|_H, \|f\|_{L^2(t_0,t_0+1;V^*)}, c_1, c_2, \Omega).
\end{equation}
This bound and the Carath\'{e}odory theorem guarantee existence of solutions for Problem $(\C{P})_n$. In a standard way, the growth condition \ref{J2} and (\ref{eq:Galerkin_bound_1}) give the bound 
\begin{equation}\label{eq:Galerkin_bound_2}
\|u_n'\|_{L^2(t_0;t_0+1;V^*)}\leq C(\|u_0\|_H, \|f\|_{L^2(t_0,t_0+1;V^*)}, c_1, c_2, \Omega).
\end{equation}
We deduce that, for a subsequence, still denoted by $n$ 
\begin{align}
& u_n \to u\ \mbox{weakly in}\ L^2(t_0,t_0+1;V) \ \mbox{and weakly-* in}\ L^\infty(t_0,t_0+1;H),\label{eq:gal_convergence_1}\\
& u'_n \to u'\ \mbox{weakly in}\ L^2(t_0,t_0+1;V^*),\\
& u_n \to u\ \mbox{strongly in}\ L^2(t_0,t_0+1;H),
\end{align}
where $u\in L^2(t_0,t_0+1;V)$ with $u'\in L^2(t_0,t_0+1;V^*)$. Denote $\Omega_T=\Omega\times (t_0,t_0+1)$. Note that, perhaps for another subsequence, we have
\begin{equation}
\label{eq:for_fatou}
u_n(x,t)\to u(x,t)\ \mbox{a.e.}\ (x,t)\in \Omega_T\ \mbox{with}\ |u_n(x,t)|\leq U(x,t), \quad U\in L^2(\Omega_T).
\end{equation}
 In a standard way it follows that $u(t_0)=u_0$. Moreover, by the growth condition \ref{J2} and the previous bounds, for yet another subsequence, we have
\begin{equation}
j_n'(x,t,u_n(x,t))\to \xi \ \mbox{weakly in}\ L^2(t_0,t_0+1;H).\label{eq:gal_convergence_5}
\end{equation}
The above convergences (\ref{eq:gal_convergence_1})-(\ref{eq:gal_convergence_5}) allow us to pass to the 
limit in (\ref{eq:galerkin_problem}) and obtain for all $v\in V$ and a.e. $t\in (t_0,t_0+1)$
$$
\langle u'(t)+Au(t) ,v\rangle + \int_{\Omega}\xi(x,t)v(x)\, \d x = \langle f(t),v\rangle.
$$
It is sufficient to show that $\xi(x,t)\in \partial j(x,t,u(x,t))$ a.e. $(x,t)\in \Omega_T$. From (\ref{eq:gal_convergence_5}) it follows that $j_n'(x,t,u_n(x,t))\to \xi$ weakly in $L^1(\Omega_T)$, and hence for $w\in L^{\infty}(\Omega_T)$ we have
$$
\int_{\Omega_T}\xi(x,t) w(x,t)\, \d x\,\d t = \lim_{n\to\infty}\int_{\Omega_T}j_n'(x,t,u_n(x,t))w(x,t)\, \d x\,\d t.
$$ 
By the growth condition \ref{J2} and by the bound (\ref{eq:for_fatou}) we can invoke the Fatou Lemma to get
$$
\int_{\Omega_T}\xi(x,t) w(x,t)\, \d x\,\d t \leq \int_{\Omega_T}\limsup_{n\to\infty,\lambda\to 0^+}\frac{j_n(x,t,u_n(x,t)+\lambda w(x,t))-j_n(x,t,u_n(x,t))}{\lambda}\, \d x\,\d t.
$$ 
We can estimate the last integrand for a.e. $(x,t)\in \Omega_T$ as follows
\begin{align*}
& \limsup_{n\to\infty,\lambda\to 0^+}\frac{j_n(x,t,u_n(x,t)+\lambda w(x,t))-j_n(x,t,u_n(x,t))}{\lambda}\\
& \qquad= \limsup_{n\to\infty,\lambda\to 0^+}\int_{-\frac{1}{n}}^{\frac{1}{n}}\varrho_n(s)\frac{j(x,t,u_n(x,t)+\lambda w(x,t)-s)-j_n(x,t,u_n(x,t)-s)}{\lambda} \, \d s\\
& \qquad\leq \limsup_{n\to\infty,\lambda\to 0^+} \sup_{s\in \left[-\frac{1}{n},\frac{1}{n}\right]}\frac{j(x,t,u_n(x,t)+\lambda w(x,t)-s)-j_n(x,t,u_n(x,t)-s)}{\lambda}\\
& \qquad\leq \limsup_{z \to u(x,t),\lambda\to 0^+} \frac{j(x,t,z+\lambda w(x,t))-j_n(x,t,z)}{\lambda} = j^0(x,t,u(x,t);w(x,t)).
\end{align*}
Hence,
$$
\int_{\Omega_T}\xi(x,t) w(x,t)\, \d x\,\d t \leq \int_{\Omega_T}j^0(x,t,u(x,t);w(x,t))\, \d x\,\d t.
$$
Since the choice of $w$ is arbitrary, from the definition of Clarke subgradient we get
$$
\xi(x,t)\in \partial j(x,t,u(x,t))\ \mbox{a.e.}\ (x,t)\in \Omega_T,
$$
and the proof is complete.
\end{proof}

The mapping $H\ni u_0\to u(t)\in H$ for $t\geq t_0$ defines a 
strict m-process in $H$. Namely, for any $t_0\leq t\in \br$ we define
\begin{equation}\label{eq:reactmulti}
U(t,t_0;u_0)=\left\{u(t): u(\cdot) \text{ is a solution to Problem $(\C{P})$ with } u(t_0)=u_0\right\}.
\end{equation}
We will now study the asymptotic behavior of the above multivalued nonautonomous dynamical system,
proving the existence of a backward bounded pullback attractor, which is fully invariant under the flow.

\subsection{Dissipative a priori estimates}

We start by the dissipative properties of our system, proving the existence of a pullback absorbing nonautonomous
set.
\begin{lemma}\label{lemma:dissipative_a_priori}
Under assumptions \ref{J1}-\ref{J3} and \ref{F1}-\ref{F2} the m-process $U:\br_d\times H\to P(H)$ defined by \eqref{eq:reactmulti} 
is monotonically dissipative.
\end{lemma}
\begin{proof}
Take $v=u(t)$ in (\ref{eq:weak_form}). We get, by \ref{J3}
$$
\frac{1}{2}\frac{\d}{\d t}\|u(t)\|_H^2 + \|u(t)\|^2 + \int_{\Omega} d_1 - d_2 |u(x,t)|^2\, \d x \leq \eps \|u(t)\|^2 + \frac{1}{4\eps}\|f(t)\|_{V^*}^2,
$$
for a.e. $t\in (t_0,\infty)$ with arbitrary $\eps>0$. We have
$$
\frac{1}{2}\frac{\d}{\d t}\|u(t)\|_H^2 + \left(1 - \frac{d_2}{\lambda_1}-\eps\right)\|u(t)\|^2 \leq  \frac{1}{4\eps}\|f(t)\|_{V^*}^2 - d_1 m(\Omega).
$$
We take $\eps=\frac{1}{2}-\frac{d_2}{2\lambda_1}$ and we obtain 
$$
\frac{\d}{\d t}\|u(t)\|_H^2 + C_1\|u(t)\|_H^2 \leq C_2\|f(t)\|_{V^*}^2 + C_3,
$$
with $C_1,C_2,C_3>0$ for a.e. $t\in (t_0,\infty)$, whence, in a standard way, it follows that for all $t\geq t_0$ we have
$$
\|u(t)\|_H^2 \leq \e^{-C_1(t-t_0)}\|u(t_0)\|_H^2 + \frac{C_3}{C_1}+C_2\e^{-C_1t}\int_{t_0}^t\e^{C_1s}\|f(s)\|_{V^*}^2\, \d s.
$$
We can assume that $t_0\leq\bar{t}$, where $\bar{t}$ is given in \ref{F2}. If $t\geq\bar{t}$ we get
$$
\e^{-C_1t}\int_{t_0}^t\e^{C_1s}\|f(s)\|_{V^*}^2\, \d s \leq \frac{1}{C_1}\  \esssup_{s<\bar{t}}\|f(s)\|_{V^*}^2 + \e^{-C_1\bar{t}}\int_{\bar{t}}^t\e^{C_1s}\|f(s)\|_{V^*}^2\, \d s:=F(t),
$$
whereas, if $t<\bar{t}$, 
$$
\e^{-C_1t}\int_{t_0}^t\e^{C_1s}\|f(s)\|_{V^*}^2\, \d s \leq \frac{1}{C_1} \esssup_{s<\bar{t}}\|f(s)\|_{V^*}^2:=F(t).
$$
Note that $F:\br\to [0,\infty)$ is a nondecreasing function, and
$$
\|u(t)\|_H^2 \leq \e^{-C_1(t-t_0)}\|u(t_0)\|_H^2 + \frac{C_3}{C_1}+C_2F(t).
$$
Define $R(t) = \sqrt{\frac{C_3}{C_1}+C_2F(t)}+1$. This function is nondecreasing. 
Moreover, the ball $B(0,R(t))$ is absorbing in $H$, and the proof is complete.
The entering time $\bar{\tau}(t,E)$ in the definition of dissipativity can be explicitly computed 
as 
$$
\bar{\tau}(t,E)=\min\left\{t-1, t-\frac{2}{C_1}\ln\|E\|_H,\bar{t}\right\}.
$$
\end{proof}

\subsection{Pullback flattening condition} 
To complete the proof of the existence of a pullback attractor, we prove that the m-process governed by  Problem $(\C{P})$ is pullback flattening.
\begin{lemma}\label{lemma:pullback_flattening}
Under assumptions \ref{J1}-\ref{J3} and \ref{F1}-\ref{F2}  the m-process $U:\br_d\times H\to P(H)$ defined by \eqref{eq:reactmulti} is pullback flattening.
\end{lemma}
\begin{proof}
For $v\in V$ define $v_1=P_mv$ and $v_2=(I-P_m)v$. We take the duality in (\ref{eq:weak_form}) with $u_2(t)$ and we get for a.e. $t> t_0$
\begin{align*}
 \frac{1}{2}\frac{\d}{\d t}\|u_2(t)\|_H^2 + \|u_2(t)\|^2 \leq & \|f(t)\|_{V^*}\|u_2(t)\|+\|\eta(t)\|_H\|u_2(t)\|_H \\
 \leq &\frac{1}{2}\|u_2(t)\|^2 + \|f(t)\|_{V^*}^2 + \frac{1}{\lambda_1}\|\eta(t)\|_H^2,
\end{align*}
with $\eta(t)\in S^2_{\partial j(\cdot,t,u(\cdot,t))}$ for a.e. $t> t_0$. From the growth condition \ref{J2}, we obtain 
$$
\|\eta(t)\|_H^2\leq 2c_1^2 m(\Omega) + 2c_2^2 \|u(t)\|_H^2,
$$
for a.e. $t>t_0$. Using the Courant-Fischer formula $\lambda_{m+1}\|u_2(t)\|_H^2\leq\|u_2(t)\|^2 $, we get
$$
\frac{\d}{\d t}\|u_2(t)\|_H^2 + \lambda_{m+1}\|u_2(t)\|_H^2 \leq 2\|f(t)\|_{V^*}^2 + \frac{4c_1^2m(\Omega)}{\lambda_1}+\frac{4c_2^2}{\lambda_1}\|u(t)\|_H^2.
$$ 
From the Translated Gronwall Lemma proved in \cite{gl-2010} (see also Lemma 3.1 in \cite{lukaszewicz_kalita}) we get for all $t\geq t_0$
\begin{align*}
\|u_2(t+2)\|_H^2 \leq & \e^{-\lambda_{m+1}}\int_t^{t+1}\|u_2(s)\|_H^2\, \d s + 2\e^{-\lambda_{m+1}(t+2)}\int_{t}^{t+2}\e^{\lambda_{m+1}s}\|f(s)\|_{V^*}^2\,\d s\\
 & + \frac{4c_1^2m(\Omega)}{\lambda_1\lambda_{m+1}}+\frac{4c_2^2}{\lambda_1}\e^{-\lambda_{m+1}(t+2)} \int_{t}^{t+2}\e^{\lambda_{m+1}s}\|u(s)\|_{H}^2\,\d s.
\end{align*}
Taking $t_0\leq \bar{\tau}(t,E)$, where $\bar{\tau}(t,E)$ is given in Lemma \ref{lemma:dissipative_a_priori}, we have $\|u_2(s)\|_H\leq \|u(s)\|_H\leq R(t+2)$ for all $s\in [t,t+2]$, with $R(t)$ for $t\in\br$ given in the proof of Lemma \ref{lemma:dissipative_a_priori}. We get
\begin{align*}
\|u_2(t+2)\|_H^2 \leq &\, \e^{-\lambda_{m+1}}R^2(t+2) + \frac{4c_2^2R^2(t+2) + 4c_1^2m(\Omega) }{\lambda_1\lambda_{m+1}}\\
 & +  2\e^{-\lambda_{m+1}(t+2)}\int_{t}^{t+2}\e^{\lambda_{m+1}s}\|f(s)\|_{V^*}^2\,\d s.
\end{align*}
For small $\delta>0$ we have
$$
\e^{-\lambda_{m+1}(t+2)}\int_{t}^{t+2}\e^{\lambda_{m+1}s}\|f(s)\|_{V^*}^2\,\d s \leq \e^{-\lambda_{m+1}\delta}\int_{t}^{t+2}\|f(s)\|_{V^*}^2\,\d s + \int_{t+2-\delta}^{t+2}\|f(s)\|_{V^*}^2\,\d s.
$$
Hence, provided the initial condition, that belongs to the set $E\in\C{B}(H)$, is taken at $t_0\leq \bar{\tau}(t-2,E)$ we get for all $t\in \br$
\begin{align*}
\|u_2(t)\|_H^2 \leq & \e^{-\lambda_{m+1}}R^2(t) + \frac{4c_2^2R^2(t) + 4c_1^2m(\Omega) }{\lambda_1\lambda_{m+1}}\\
 & +  2\e^{-\lambda_{m+1}\delta}\int_{t-2}^{t}\|f(s)\|_{V^*}^2\,\d s + 2\int_{t-\delta}^{t}\|f(s)\|_{V^*}^2\,\d s.
\end{align*}
Having $t, \eps$ fixed we take $\delta>0$ sufficiently small such that $2\int_{t-\delta}^{t}\|f(s)\|_{V^*}^2\,\d s\leq \frac{\eps}{4}$. Next, we pick $m$ sufficiently large such that
$$
\e^{-\lambda_{m+1}}R^2(t) \leq \frac{\eps}{4},\ \ \frac{4c_2^2R^2(t) + 4c_1^2m(\Omega) }{\lambda_1\lambda_{m+1}}\leq \frac{\eps}{4},\ \  2\e^{-\lambda_{m+1}\delta}\int_{t-2}^{t}\|f(s)\|_{V^*}^2\,\d s\leq \frac{\eps}{4}.
$$
The proof is complete.
\end{proof}

As a consequence of Lemmata \ref{lemma:dissipative_a_priori} and \ref{lemma:pullback_flattening} we get, by Theorem \ref{thm:metricequiv2},
the following theorem on the existence of a pullback attractor.

\begin{theorem}\label{thm:attr_exists}
Under assumptions \ref{J1}-\ref{J3}  and \ref{F1}-\ref{F2}  the m-process $U:\br_d\times H\to P(H)$ defined by \eqref{eq:reactmulti} has a backward bounded pullback attractor.
\end{theorem}

\subsection{Invariance of the pullback attractor} 
As we have see in the previous section, the invariance of the pullback attractor follows from the graph of $U$ being closed.
We do not use Corollary \ref{corollary:inv} in its full strength here, as the strict process $U$ happens to be $t_\star$-closed for any
arbitrary $t_\star>0$.

\begin{lemma}\label{lemma:inv}
Under assumptions \ref{J1}-\ref{J3}  and \ref{F1}-\ref{F2}   the m-process $U:\br_d\times H\to P(H)$ is closed.
\end{lemma}
\begin{proof}
Let $u_n(t_0) = u_{0,n}\to u_0$ strongly in $H$. Let moreover $\{u_n\}_{n=1}^\infty$ be a sequence of trajectories with initial conditions $u_{0,n}$ and corresponding selections of $S^2_{\partial j(\cdot,t,u_n(\cdot,t))}$ given by $\eta_n$. Choose $t>t_0$ and assume that $u_n(t)\to w$ strongly in $H$. From a priori estimates analogous to (\ref{eq:Galerkin_bound_1})-(\ref{eq:Galerkin_bound_2}) it follows that $u_n$ is bounded in $L^2(t_0,t;V)\cap L^\infty(t_0,t;H)$ and $u_n'$ is bounded in $L^2(t_0,t;V^*)$. By the growth condition \ref{J2} and the bound on $u_n$ it follows that $\eta_n$ is bounded in $L^2(t_0,t;H)$. Hence, for a subsequence still denoted by $n$, we have
\begin{align*}
& u_n \to u\ \mbox{weakly in}\ L^2(t_0,t;V) \ \mbox{and weakly-$*$ in}\ L^\infty(t_0,t;H),\\
& u'_n \to u'\ \mbox{weakly in}\ L^2(t_0,t;V^*),\\
& \eta_n \to \eta\ \mbox{weakly in}\ L^2(t_0,t;H).
\end{align*}
Integrating from $t_0$ to $t$ the identity 
\begin{equation}
\label{eq:convergence_closedness}
u_n(s)-u(s) = u_{0,n}-u(t_0) + \int_{t_0}^s u_n'(r)-u'(r) \d r,
\end{equation}
which is valid in $V^*$ for all $s\in [t_0,t]$, and taking the duality with $v\in V$ we get
$$
\int_{t_0}^t ( u_n(s)-u(s),v)\, \d s = (t-t_0) (u_{0,n}-u(t_0),v) + \int_{t_0}^t \int_{t_0}^s \langle u_n'(r)-u'(r),v\rangle\, \d r\, \d s.
$$
Passing to the limit $n\to\infty$, both integrals must converge to zero and hence we have $u_{0,n}\to u(t_0)$ weakly in $V^*$ and moreover $u(t_0)=u_0$.
Coming back to (\ref{eq:convergence_closedness}) we take duality with $v\in V$ and pass to the limit 
again which yields $u_n(s)\to u(s)$ weakly in $V^*$ for all $s\in [t_0,t]$. In particular $u(t)=w$. It remains 
to prove that $(u,\eta)$ satisfy (\ref{eq:weak_form}) and (\ref{eq:selection}) on the interval $(t_0,t)$. 
Indeed, we can pass to the limit in the equation
$$
u_n'(s)+Au_n(s)+\eta_n(s)=f(s)\ \ \text{in}\ \ V^*\ \ \text{a.e.}\,\, s\in(t_0,t),
$$
which yields 
$$
u'(s)+Au(s)+\eta(s)=f(s)\ \ \text{in}\ \ V^*\ \ \text{a.e.}\,\, s\in(t_0,t).
$$
Since, for a possibly further subsequence, $u_n(x,s)\to u(x,s)$ a.e. $(x,s)\in \Omega\times (t_0,t)$ 
and $\eta_n\to\eta$ weakly in $L^1(\Omega\times (t_0,t))$,
it follows from the Aubin and Cellina convergence theorem (see \cite{aubin-frankowska}*{Theorem 7.2.2 }) 
that $\eta\in \partial j(x,s,u(x,s))$ a.e. $(x,s)\in \Omega\times (t_0,t)$ and the assertion follows. The proof is complete.
\end{proof}

From Lemma \ref{lemma:inv}, Theorem \ref{thm:attr_exists}, and Corollary \ref{corollary:inv} we 
get the following result on the invariance of the pullback attractor for Problem ($\C{P}$).

\begin{theorem}\label{thm:attr_invariant}
Under assumptions \ref{J1}-\ref{J3}  and \ref{F1}-\ref{F2}   the m-process $U:\br_d\times H\to P(H)$ 
defined by \eqref{eq:reactmulti} has a backward bounded and invariant pullback attractor.
\end{theorem}

\section{Point-dissipativity and m-semiflows}\label{sec:point}

\noindent Although the picture on the existence of pullback attractors for dissipative m-processes is 
quite complete, we can think of weakening the notion of dissipativity to the
so-called \emph{point-dissipativity}. In which case, the analysis has to
be carried out in a much more careful way, and some sort of continuity
assumptions on the m-process have to be taken into account, even if we do not 
require the attractor to be invariant in the first place. We dedicate
this last section to the discussion of this particular notion of dissipativity. We restrict our 
analysis to the case of \textit{autonomous} dynamical systems, considering, instead of
\emph{m-processes}, their corresponding autonomous version known as \emph{m-semiflows}.

For the general concepts regarding point-dissipativity and nonautonomous single-valued processes,
we refer the reader to the work \cite{caraballo-2010}, in which the existence of a pullback attractor 
is proven under the following requirements:
\begin{itemize}
\item for  given $t\in\mathbb{R}$ and $\tau>0$, the family $\{U(s,s-\tau;\cdot)\, :\, s\leq t\}$ is assumed to be \textit{equicontinuous};
\item the process $U$ must be \textit{pullback strongly asymptotically compact} (see \cite{caraballo-2010}*{Definition 2.9});
\item the process $U$ must be pullback strongly bounded (see \cite{caraballo-2010}*{Definition 2.6}).
\end{itemize}
It is, to the best our knowledge, still an open problem to  study the existence of pullback attractors of 
m-processes when only pullback point-dissipativity is assumed (in place of pullback dissipativity as in Definition \ref{dissipative}). 
Note that point-dissipativity is a natural notion when considering the case of gradient systems 
as explained in  in \cite{carvalho-book}*{Chapter 2.5.1}. In the sequel we will recall several theorems on the 
existence of global attractor for point-dissipative multivalued autonomous dynamical systems and 
we will present two examples which show that, in contrast to dissipative dynamical systems, continuity 
is needed for the existence of global attractor even in the minimal sense. 

\subsection{Multivalued autonomous dynamical systems}
As previously mentioned, we will consider here 
the case of multivalued autonomous system, for which the evolution depends only on
the difference $t-\tau$. In other words, the equality
$$
U(t,\tau;\cdot)=U(t-\tau,0;\cdot)
$$
holds for every $t\geq \tau$, and the one-parameter family of 
set-valued maps
$$
S(t;\cdot)=U(t,0;\cdot):X\to P(X), \qquad t\geq 0,
$$
fulfills the semiflow properties
\begin{itemize}
\item $S(0;x)=\{x\}$, for all $t\geq 0$, $x\in X$;
\item $S(t+\tau;x)\subset S(t;S(\tau;x))$ for all $t,\tau\geq 0$, $x\in X$.
\end{itemize}
Such family will be called \textit{multivalued semiflow}, or \textit{m-semiflow}. If, instead of the inclusion $S(t+\tau;x)\subset S(t;S(\tau;x))$, we have the equality $S(t+\tau;x)= S(t;S(\tau;x))$, then we say that the semiflow is \textit{strict}. We use again the standard notation 
$$
S(t;B)=\bigcup_{x\in B}S(t;x)
$$ 
for $B\subset X$ and $t\geq 0$. If, for an m-semiflow $S$, the set $S(t;x)$ is a singleton for all 
$(t,x)\in [0,\infty)\times X$, then we say that $S$ is a \textit{semigroup}. 
 Note that every semigroup is automatically strict.
We recall here the definitions of dissipativity, point-dissipativity, 
asymptotic compactness and global attractor for m-semiflows.

\medskip

\noindent \textbf{Dissipativity.} 
 A set $B_0\in \C{B}(X)$ is \emph{absorbing} if for every $B\in \C{B}(X)$ there exists an entering time 
 $t_B>0$ such that
	$$
	S(t;B)\subset B_0, \qquad \forall t\geq t_B.
	$$
	In this case, $S(t;\cdot)$ is said to be \emph{dissipative}.

\medskip

\noindent \textbf{Point-Dissipativity.} 
 A set $B_0\in\C{B}(X)$ is \emph{point-absorbing} if for every  $x\in X$ there exists an entering time $t_x>0$ such that
	$$
	S(t;x)\in B_0, \qquad \forall t\geq t_x.
	$$
	In this case, $S(t;\cdot)$ is said to be \emph{point-dissipative}. 
	
\medskip

\noindent \textbf{Asymptotic compactness.}
A semigroup $S(t;\cdot)$ is said to be \emph{asymptotically compact} if the following
holds true:  for any bounded set $B\subset X$  and any sequences  $x_n\in B$,  $t_n\to\infty$,
and $y_n\in S(t_n;x_n)$, there exist $y\in X$  and a subsequence $y_{n_k}\to y$.

\medskip

Following \cite{chep-conti-pata-2012} and our definition of pullback attractor, we will use the following definition of global attractor for m-semiflow.

\medskip

\noindent \textbf{Global attractor.} 
A compact set $A\subset X$ is a global attractor
for an m-semiflow $S$ if it attracts every bounded set in $X$, i.e. if for and $B\in \C{B}(X)$ we have
$$\lim_{t\to\infty}\mbox{dist}(S(t;B),A)=0,$$
and $A$ is the minimal compact set that has this property.

\subsection{Global attractors for m-semiflows}
Observe that an m-semiflow is asymptotically compact if and only if the associated m-process 
is pullback asymptotically compact. Moreover, an m-semiflow is dissipative if and only if the 
associated m-process is monotonically dissipative. Finally, the set $A$ is a global attractor for 
the m-semiflow $S$ if and only if the nonautonomous set
 $\mathbb{A}=\{A(t)\, :\, t\in\mathbb{R}\}$, given by $A(t)=A$ for all $t\in \mathbb{R}$ is the 
 global attractor for the associated m-process $U$. 
As a consequence of Theorem \ref{thm:eqivalence_attractor} and Propositions \ref{prop:nonstrictinv}-\ref{prop:strictinv}, we have the following theorem on 
the existence of global attractors for m-semiflows.

\begin{theorem}\label{thm:last_1}
Let $S$ be an m-semiflow on a complete metric space $X$. $S$ has a 
global attractor $A$ if and only if $S$ is dissipative and asymptotically compact. Moreover:
\begin{itemize}
\item if $S(t_\star;\cdot)$ has closed graph for some $t_\star>0$, then $A\subset S(t_\star;A)$;
\item if $S(t;\cdot)$ has  closed graph for all $t>0$, then $A\subset S(t;A)$ for all $t>0$;
\item if $S$ is strict and if $S(t_\star;\cdot)$ has closed graph for some $t_\star>0$, then $A= S(t;A)$ for all $t>0$.
\end{itemize}
\end{theorem}

The above theorem can be viewed as a generalization of  \cite{chep-conti-pata-2012}*{Theorem 12} 
 to the case of m-semiflows. Clearly, conditions such as $\omega$-limit compactness or in terms of
 the Kuratowski measure can be included in the above statement, in the same spirit as Theorem \ref{thm:metricequiv}.
Also, it recovers some already known results present in \cites{Coti_Zelati_set_valued,melnik-1998}.
 We now turn to point-dissipative semiflows. A general result is available in \cite{melnik-1998}.
 \begin{theorem}\label{thm:last_3}
Let $X$ be a complete metric space and let $S$ be an m-semiflow such that the map $S(t;\cdot):X\to P(X)$ 
is upper-semicontinuous and has closed values for all $t\geq 0$. Then 
$S$ has a global attractor $A$ if and only if $S$ is point-dissipative and asymptotically compact. Moreover,
$A\subset S(t;A)$ for all $t>0$.
\end{theorem}
The negative invariance above follows from the fact that an m-semiflow which is upper semi-continuous and has closed
valued necessarily satisfies the closed graph condition. The  surprising fact is that, in the strict case, much
weaker requirements imply the same conclusion \cite{Coti_Zelati_set_valued}*{Theorem 4.6}.

 \begin{theorem}\label{thm:last_2}
Let $X$ be a complete metric space and let $S$ be a strict m-semiflow such that the map $S(t_\star;\cdot):X\to P(X)$ 
has closed graph for certain $t_\star>0$. Then $S$ has a global attractor if and only if $S$ is point-dissipative and 
asymptotically compact. Moreover, $A= S(t;A)$ for all $t>0$.
\end{theorem}
The above results yields some consequences in the theory of gradient systems, as point-dissipativity is basically 
encoded in the existence of a Lyapunov functional. In particular, classical theorems hold in more general settings
and the proofs are much simpler (see \cite{Coti_Zelati_set_valued}*{Section 6}).
 
\subsection{Dissipativity vs. point-dissipativity}
The power of the strictness condition is perhaps highlighted in the  following consequence of Theorem \ref{thm:last_2},
which we state here as a corollary.
\begin{corollary}\label{thm:CZ}
Let $S(t;\cdot)$ be a strict and asymptotically compact m-semiflow on a complete metric space $X$, and assume that
there exists $t_\star>0$ such that the  map $S(t_\star;\cdot):X\to P(X)$ has closed graph.
Then 
$S(t;\cdot)$ is dissipative if and only if it is point-dissipative. 
\end{corollary}
This is somewhat counterintuitive at first. Indeed,
it is obvious that a dissipative m-semiflow is also point-dissipative. The converse is not true, 
even in Hilbert spaces and with	linear and continuous semigroups of contractions.

\begin{example}
Consider $X=\ell^2$, the Hilbert space of infinite, square summable, real-valued sequences. Let $\sigma:X\to X$ 
be the left shift, defined by
$$
\sigma((x_1,x_2,x_3,\ldots)=(x_2,x_3,\ldots).
$$
A discrete semigroup may be defined by taking 
$$
S(n;\cdot)=\sigma\circ\sigma\circ\ldots\circ \sigma, \qquad n \text{ times}, n\in\bn.
$$
Clearly, for every $n\in\bn$, $S(n;\cdot)$ is continuous and linear, since $\sigma$ is. Moreover,
$$
\|S(n;\boldsymbol{x})\|\leq\|\boldsymbol{x}\|, \qquad \forall \boldsymbol{x}\in \ell^2. 
$$
\noindent$\diamond$  \textbf{$S(n;\cdot)$ is point-dissipative:} as a point-absorbing set, we can
take the unit ball of $\ell^2$. Fix $\boldsymbol{x}\in \ell^2$. By definition, there exists $N_x>0$
such that
$$
\sum_{i=N_x}^\infty |x_i|^2\leq 1.
$$
As a consequence,
$$
\|S(n;x)\|^2\leq \sum_{i=N_x}^\infty |x_i|^2\leq 1, \qquad \forall n\geq N_x,
$$
proving the claim.

\medskip

\noindent$\diamond$  \textbf{$S(n;\cdot)$ is \emph{not} dissipative:} for any $R>0$ fixed, we will prove that
the ball of radius $R$, which we call $B_R$, is not absorbing. Consider the sequence $\{\boldsymbol{\e}_k\}\subset \ell^2$
of vectors that are all zero except for the $k$-coordinate, and define $\boldsymbol{x}_k=2R \,\boldsymbol{\e}_{k+1}$. Notice that
$\|\boldsymbol{x}_k\|=2R$, so $\{\boldsymbol{x}_k\}$ is a bounded set. However, 
$$
\| S(n;\boldsymbol{x}_n)\|=2R, \qquad \forall n\in \bn,
$$
and therefore $\{\boldsymbol{x}_k\}$ is never going to be absorbed by $B_R$.
\end{example}

One feature missing in the above example is some sort of compactness of the m-semiflow or of the 
ambient space. 
For instance, if $X$ were a compact metric space, then an m-semiflow is
dissipative if and only if it is point-dissipative. Of course, this is a trivial case, since
the whole space can be taken as an absorbing set. 
Also, from Corollary \ref{thm:CZ} we see that if we add asymptotic compactness and a very weak continuity property to the m-semiflow,
the two notions of dissipativity are completely equivalent.
The question is whether the assumption that the m-semiflow has closed graph for some $t_\star>0$
is really needed or not. In other words, for an arbitrary strict and asymptotically compact m-semiflow, does
Proposition \ref{thm:CZ} hold? The answer is negative, even in the single valued case.

\begin{example}
Let $Z=[0,1]\times [0,\infty)$ with the usual euclidean metric.
For $z_0=(x_0,y_0)\in Z$, define $z(t)=(x(t),y(t))=S(t;z_0)$ in the following way.
\begin{itemize}
	\item If $x_0=0$, then
	\begin{equation}\label{eq:1}
	x(t)=0, \qquad 
	y(t)=\begin{cases}
	y_0-t, \quad &t\in[0,y_0],\\
	0, \quad &t>y_0.
	\end{cases}
	\end{equation}
	\item If $x_0>0$, then
	\begin{equation}\label{eq:2}
	x(t)=x_0, \qquad 
	y(t)=\begin{cases}
	y_0-x_0t, \quad &t\in[0,y_0/x_0],\\
	0, \quad &t>y_0/x_0.
	\end{cases}
	\end{equation}	
\end{itemize}

\noindent$\diamond$  \textbf{$S(t;\cdot)$ is point-dissipative:} if we fix any $z_0\in Z$, we have $z(t)=0$  for every $t\geq t_{z_0}$ for some $t_{z_0}>0$. Therefore,the set $B_0=[0,1]\times \{0\}$ is point-absorbing. In fact, any bounded set of the form
$B_0=[0,1]\times [0,M]$ with $M\geq 0$ would do.

\medskip

\noindent$\diamond$  \textbf{$S(t;\cdot)$ is asymptotically compact:} given any bounded set $B\subset Z$, we can
find $M_B>0$ so that $S(t;B)\subset [0,1]\times [0,M_B]$ for every $t\geq 0$. As $Z$ is finite-dimensional, 
this implies asymptotic compactness at once.

\medskip

\noindent$\diamond$  \textbf{$S(t;\cdot)$ is \emph{not} dissipative:} notice first that it suffices to prove that for no $M>0$, the set $B_M=[0,1]\times [0,M]$ is absorbing. To prove that $B_M$ is not absorbing, consider the bounded set $K_M=[0,1]\times \{2M\}$. We want to 
prove that $K_M$ is not absorbed by $B_M$. Let $x_{0,n}=1/n$, for $n\geq 1$. Then, the corresponding sequence
$z_n(t)=(x_n(t),y_n(t))=S(t;(x_{0,n},2M))$ takes the form
	$$
	x_n(t)=\frac{1}{n}, \qquad 
	y_n(t)=\begin{cases}
	2M-t/n, \quad &t\in[0,2M n],\\
	0, \quad &t>2M n.
	\end{cases}
	$$	
Let $t_n:= Mn/2$. It is easy to check that $z_n(t_n)\notin B_M$ for every $n$. 
In view of the fact that $t_n\to \infty$, the dissipativity condition is violated. Hence, $B_M$ is not absorbing. 

\medskip

\noindent$\diamond$  \textbf{$S(t;\cdot)$ is closed for no $t>0$:} let $t>0$ be arbitrarily fixed. For $n\geq 1$, consider
the sequence $z_{0,n}=(1/n,y_0)$ with $y_0>0$. Clearly, $z_{0,n}\to (0,y_0)$. Also, 
$$
z_n(t)=S(t;z_{0,n})=(1/n, y_0 - t/n )
$$
for $n$ sufficiently large. In particular, $z_n(t)\to (0,y_0)$ as $n\to \infty$. However, from \eqref{eq:1} it is apparent that
$$
(0,y_0)\neq S(t;(0,y_0)) \qquad \forall t>0,
$$
concluding the proof of the claim.
\end{example}
It is not hard to see that, in the above example, the function $z(t)=(x(t),y(t))$ solves the ordinary differential equation
\begin{align*}
x'(t)=0,\qquad y'(t)=
\begin{cases}
0, \qquad &\text{if } y(t)=0,\\
-1, \qquad &\text{if } x(t)=0 \text{ and } y(t)>0,\\
-x(t), \qquad &\text{if } x(t)>0 \text{ and } y(t)>0,
\end{cases}
\end{align*}
where the solution is understood in Carath\'{e}odory sense, i.e. we seek for an absolutely 
continuous function such that the above equations hold for a.e. $t\geq 0$. In the next paragraph, 
we also provide an example of a semigroup of solution operators which arises from a partial differential
equation. 

\subsection{A PDE with nowhere closed solution semigroup}
Consider the spaces $V=H^1_0(0,\pi)$ and $H=L^2(0,\pi)$, and denote by $(\cdot,\cdot)$ the scalar product in $H$. 
Let $v^1$ be the first eigenfunction of the operator $-\partial_{xx}$ with the
Dirichlet boundary conditions given in the definition of $V$. Then $v^1(x)=\sin(x)$ for $x\in (0,\pi)$ 
and the corresponding eigenvalue is given by $\lambda_1=1$. 
Moreover $\|v^1\|_V^2= \|v^1\|_H^2 = \pi/2$. 
Denote by $P:H\to H_1$ the orthogonal projection operator onto the one dimensional space 
spanned by the function $v^1$. For any $u\in H$ we have $Pu=\alpha(u) v^1$ with 
$\alpha(u)\in \mathbb{R}$ given as $\alpha(u) = \frac{2}{\pi}(u,v^1)$. Now denote 
$H_+=\{u\in H\, : \alpha(u) >0 \}$ and $H_-=\{u\in H\, : \alpha(u) \leq 0 \}$.
Consider the following two problems:

\medskip

\noindent \textbf{Problem ($+$).} Find $u\in L^2_{loc}(\mathbb{R}^+;V)$ with 
$u_t\in L^2_{loc}(\mathbb{R}^+;V^*)$ such that for a.e. $t\in \mathbb{R}^+$ and all $v\in V$ we have
$$
\int_0^\pi u_t(x,t)v(x)\, \d x + \int_0^\pi u_{x}(x,t)v_{x}(x)\, \d x = 0.
$$

\medskip

\noindent \textbf{Problem ($-$).} Find $u\in L^2_{loc}(\mathbb{R}^+;V)$ with 
$u_t\in L^2_{loc}(\mathbb{R}^+;V^*)$ such that for a.e. $t\in \mathbb{R}^+$ and all $v\in V$ we have
$$
\int_0^\pi u_t(x,t)v\, \d x + \int_0^\pi \big[u_{x}(x,t)+v^1_{x}(x)\big]v_{x}(x)\, \d x = 0.
$$
The first of the above two problems consists in solving the homogeneous heat equation 
$u_t=u_{xx}$, while the second one in solving the heat equation with a source term 
$u_t=u_{xx}+v^1_{xx}$, both with homogeneous Dirichlet boundary conditions. As it is well known,
both problems have unique solutions and the associated semigroups are well defined. 
They will be denoted, respectively, by $S_+:\mathbb{R}^+\times H\to H$ and $S_-:\mathbb{R}^+\times H\to H$. 

Let us take $Pu(t)=\alpha(u(t))v^1$ as the test function in Problem ($+$). It is easy to see that we obtain
$$
\alpha'(u(t))+\alpha(u(t))=0
$$ 
and hence $\alpha(u(t))=\alpha(u(0))\e^{-t}$. This means that 
the region $H_+$ is positively invariant with respect to $S_+$ (i.e. $S_+(t;H_+)\subset H_+$ 
for all $t\geq 0$). Analogously, the same test function  in 
Problem ($-$) yields  
$$
\alpha'(u(t))+\alpha(u(t)) = -1
$$ 
and hence $\alpha(u(t))=\alpha(u(0))\e^{-t}-1+\e^{-t}$. This means that the region $H_-$ 
is positively invariant with respect to $S_-$ (i.e. $S_-(t;H_-)\subset H_-$ 
for all $t\geq 0$). We can define the semigroup
$$
S(t;u_0)=\begin{cases}
S_+(t;u_0)\ \ \text{for}\ \ u_0\in H_+,\\
S_-(t;u_0)\ \ \text{for}\ \ u_0\in H_-.
\end{cases}
$$
The semigroup $S$ defines the solutions of the initial and boundary value problem for the PDE
$$
u_t(t,x)=u_{xx}(t,x)+v^1_{xx}(x)\chi_{H_-}(u(x,t))\ \ \text{on}\ \ [0,\infty)\times[0,\pi],
$$
with homogeneous Dirichlet boundary conditions. In the above formula the expression $\chi_K$ 
denotes the characteristic function of the set $K$ given by 
$$
\chi_K(v)=\begin{cases}
1\ \ \text{for}\ \ v\in K,\\
0\ \ \text{otherwise}.
\end{cases}
$$
Of course the problem is non-local, as the quantity $\chi_{H_-}(u(x,t))$ depends on 
the values of $u(x,\cdot)$ on the whole interval $(0,\pi)$. Moreover, this function is discontinuous. Indeed 
if $v_n\in H_+$ and $v_n\to 0$ strongly in $H$, then $0 = \chi_{H_-}(v_n) \not\to \chi_{H_-}(0) = 1$.
Nonetheless, the longtime behavior of $S(t;\cdot)$ can be studied in the usual way.

\medskip

\noindent$\diamond$  \textbf{$S(t;\cdot)$ is dissipative:} by taking the test function $v=u(t)$ in Problem ($+$)  
we find that 
$$
\|S_+(t;u)\|_H\leq \|u\|_H\e^{-t}
$$ 
for $u\in H$ and $t\geq 0$ and, similarly, by 
taking $v=u(t)+v^1$ in Problem ($-$) we get 
$$
\|S_-(t;u)+v^1\|_H\leq \|u+v^1\|_H\e^{-t}
$$ 
for 
$u\in H$ and $t\geq 0$ (in fact Problem ($-$) is simply Problem ($+$) shifted by $-v^1$). 
Hence, all trajectories of Problem ($+$) converge to zero, while all trajectories of 
Problem ($-$) converge to $-v^1$.  For $\eps>0$, define the sets $N_\eps = B_H(0,\eps)\cup B_H(-v^1,\eps)$. 
If $B\in \C{B}(H)$ and $\|B\|_H:=\sup_{u\in B}\|u\|_H$ then
$$
t\geq t_B=2\ln \frac{\|B\|_H+\|v^1\|_H}{\eps} \quad\implies \quad S(t;B)\subset N_\eps.
$$
In particular, $N_\eps$ is absorbing for every $\eps>0$.
\medskip

\noindent$\diamond$  \textbf{$S(t;\cdot)$ is asymptotically compact:} 
this follows from the well known fact that 
both $S_+$ and $S_-$ are compact (see for example \cite{robinson}).

\medskip

\noindent In view of Theorem \ref{thm:last_1}, the above two facts yield the existence of the global attractor 
for the semigroup $S$. 
This attractor is not connected and, in fact, $A=\{0,-v^1\}$ consists of two points. Since $0\in H_-$,  
the point $0$ is attracted by $-v_1$ and hence the attractor is neither positively, nor negatively 
invariant. By Theorem \ref{thm:last_1}
the semigroup $S$ cannot be $t_\star$-closed. Indeed, it is not. Let the sequence $\{u_n\} \subset H_+$ 
converge strongly in $H$ to $u_\infty$ such that $\alpha (u_\infty)=0$ and let $S(t_\star;u_n)\to z$ 
strongly in $H$. Since the region $H_+$ is positively invariant with respect to $S_+$ we must have $\alpha(S(t_\star;u_n))>0$, and, by the continuity of $\alpha$, we 
must have $\alpha(z)\geq 0$. But, since $u_\infty \in H_-$, it follows that 
$\alpha(S(t_\star;u_\infty))=-1+\e^{-t_\star}<0$. This means that $S(t_\star;u_\infty)\neq z$ and, in turn, 
$S$ cannot be $t_\star$-closed.

\section*{Acknowledgments}
\noindent We would like to thank Grzegorz {\L}ukaszewicz for his useful remarks during the preparation of the paper. 
MCZ was partially supported by the National Science Foundation 
under the grant NSF DMS-1206438 and by the Research Fund of Indiana University.
PK was supported by a Marie Curie International Research Staff Exchange Scheme 
Fellowship within the seventh European Community Framework Programme under 
Grant Agreement no. 2011-295118, by the International Project co-financed by the 
Ministry of Science and Higher Education of Republic of Poland under grant no. W111/7.PR/2012, and by
the National Science Center of Poland under Maestro Advanced
Project no. DEC-2012/06/A/ST1/00262.



\end{document}